%% file: Gomilko_Kosowicz_Tomilov_Revised.tex
\documentclass[11pt]{amsart}

\usepackage{amsmath,amscd,amssymb}
\usepackage[mathscr]{eucal}

\input{hatomakro}

 \newtheorem{thm}{Theorem}[section]
 \newtheorem{cor}[thm]{Corollary}
 \newtheorem{lemma}[thm]{Lemma}
 \newtheorem{prop}[thm]{Proposition}

 \theoremstyle{definition}

 \theoremstyle{remark}
 \newtheorem{rem}[thm]{Remark}

\newtheorem{example}[thm]{Example}

\newtheorem{remark}[thm]{Remark}

\numberwithin{equation}{section}

\numberwithin{equation}{section} \numberwithin{theorem}{section}

\begin{document}

\title[
An approach to approximation theory of semigroups
 ]
{
A general approach to approximation theory of operator semigroups
}

\author{Alexander Gomilko}
\address{Faculty of Mathematics and Computer Science\\
Nicolas Copernicus University\\
ul. Chopina 12/18\\
87-100 Toru\'n, Poland
}

\email{alex@gomilko.com}



\author{Sylwia Kosowicz}
\address{Faculty of Mathematics and Computer Science\\
Nicolas Copernicus University\\
ul. Chopina 12/18\\
87-100 Toru\'n, Poland
}

\email{sylwiak@mat.umk.pl}

\author{Yuri Tomilov}
\address{
Institute of Mathematics\\
Polish Academy of Sciences\\
\'Sniadeckich 8\\
00-956 Warsaw, Poland
}

\email{ytomilov@impan.pl}

\thanks{This work was completed with the support of the NCN grant
 2014/13/B/ST1/03153.
 It was also partially supported by the  grant 346300 for IMPAN from the Simons Foundation and the matching 2015-2019 Polish MNiSW fund.}



\date{\today}

\begin{abstract}
By extending  ideas from \cite{GTJFA}, we develop a general, functional calculus approach to approximation
of
$C_0$-semigroups
on Banach spaces
by  bounded completely monotone functions of their generators.
The approach comprises  most of well-known approximation formulas,
yields optimal convergence rates, and sometimes even leads to sharp constants.
In an important particular case when semigroups are holomorphic, we are able
to significantly improve our results for general semigroups.
Moreover, we present several second order approximation formulas with rates,
which in such a general form appear in the literature for the first time.

\end{abstract}

\maketitle

\section{Introduction}
The approximation formulas for  $C_0$-semigroups
play an important role in the study of PDEs and their numerical analysis. They are also
indispensable in probability theory and in the theory of approximation of functions,
and helpful in the study of semigroups themselves.
Basic results on semigroup approximation
can be found
 e.g. in \cite[Chapters III.4, III.5]{EngNag}.
While some  earlier developments are thoroughly described in \cite{Butzer} and \cite{Kupcov},
it seems there's no a modern and comprehensive account on approximation of $C_0$-semigroups and related matters.
Some aspects of applications of semigroup approximations, although different from the setting of this paper,
 are presented in \cite{Altomare1} and \cite{Altomare}.

For quite a long time, the approximation theory of $C_0$-semigroups existed
as a number of distant formulas, and each of the formulas required a separate theory. A  typical illustration  of that phenomena
 is the following famous Yosida's formula with rates, a subject of several papers
 (see e.g. \cite{GTJFA} for a relevant discussion).
Recall that if  $-A$ is the generator of a bounded $C_0$-semigroup
$(e^{-tA})_{t \ge 0}$ on a Banach space $X,$ and  $\alpha \in
(0,2],$ then for all $x\in \dom(A^\alpha),\quad t>0,$ and $n \in
\mathbb N,$
\[
\|e^{-n t A(n+A)^{-1}}x - e^{-tA}x\| \le 16 M
\left(\frac{t}{n}\right)^{\alpha/2} \|A^\alpha x\|,
\]
where $M:=\sup_{t \ge0} \|e^{-tA}\|.$ Moreover, if $(e^{-tA})_{t \ge 0}$ is holomorphic and sectorially bounded,  then
for all $\alpha \in [0,1],$ $x\in \dom(A^\alpha),\quad t>0,$ and $n \in
\mathbb N,$
\[
\|e^{-n t A(n+A)^{-1}}x - e^{-tA}x\| \le C
(nt^{1-\alpha})^{-1} \|A^\alpha x\|,
\]
where  $C >0$ is an absolute constant.
Replacing Yosida's approximating family $(e^{-n t A(n+A)^{-1}})_{n \ge 1}$ by another family of functions of $A,$ e.g. by Euler's approximation $((1+ tA/n)^{-n})_{n \ge 1}$,
one can usually get a statement of a similar nature. However technical details are often annoyingly different,
  and it is natural to ask whether there is a way
 to unify all of those separate considerations.
One of the aims of this paper is to show that the functional calculi theory offers
a general and fruitful point of view at semigroup approximations.

 A  functional calculus approach to approximation of $C_0$-semigroups,
 with an emphasis on rational approximations, has been initiated
 in the ground-laying papers  \cite{BT79} and
\cite{HK79}, and developed further in many subsequent articles.
For comparatively recent contributions to this area, based on functional calculus ideology,
 see  e.g. \cite{EgRo},  \cite{Hassan}, \cite{Flory}, \cite{Grimm}, \cite{JaNe12}, \cite{Kovacs07} and \cite{Ne13}.
A fuller discussion with some other relevant references can be found in \cite{GTJFA}.

At the same time, using ideas from probability
theory and probabilistic representation formulas, a unifying approach to  approximation theorems  for $C_0$-semigroups
 was proposed  by Butzer and Hahn \cite{Butzer} and  Chung \cite{Chung}.
Somewhat closed (but avoiding an explicit probability language) studies of Dunford-Segal (or Hille) and Yosida's formulas
were done by Ditzian, \cite{D69}-\cite{D71}.
The latter papers contained also a thorough discussion of convergence rates.
A comprehensive probabilistic approach has been further developed
 in a series of papers by Pfeifer, \cite{Pf84a}-\cite{Pf93}, see also \cite{Fang}.
Pfeifer combined a number of well-known approximation formulas in terms of several probabilistic ones
and derived  the corresponding approximation  rates (optimal according to \cite{D71} and \cite{Fang}).
More comments on his approach will be given below.

In this paper, we take a different route, and
elaborate further the functional calculus techniques for approximation
of $C_0$-semigroups on Banach spaces.
Being unaware of Pfeifer's work, in \cite{GTJFA},
the first and the third author put
several well-known approximation formulas for $C_0$-semigroups
into the framework  of exponentials of (negative) Bernstein functions and the Hille-Phillips functional calculus.
Recall that by Bernstein's theorem a nonnegative  function $\varphi$ on $[0,\infty)$ is  Bernstein if and only if there exists a weak star continuous convolution semigroup of Borel measures $(\mu_t)_{t \ge 0}$
on $[0,\infty)$ such that
\begin{equation}\label{bernst}
e^{-t\varphi(z)}=\int_{0}^{\infty}e^{-sz}\, \mu_t(ds), \qquad t \ge 0.
\end{equation}
 for all $z \ge 0.$ It is well-known that if $\varphi$ is Bernstein, then $\varphi$ extends analytically to the (open) right half-pane $\mathbb C_+$ and continuously to $\overline{\mathbb C}_+,$ so that \eqref{bernst} holds in fact for all $z \in \overline{\mathbb C}_+.$
 with ${\rm Re} \, z\ge 0.$ One of the basic observations
in  \cite{GTJFA} is that many approximation formulas for $C_0$-semigroups
follow from either of  asymptotic relations
\begin{align}
e^{-tz} - e^{-n\varphi(tz/n)}\to 0,&\qquad n \to \infty, \label{AsB}\\
e^{-tz} - e^{-nt\varphi(z/n)} \to 0,& \qquad n\to \infty, \label{AsB1}
\end{align}
for all $t>0$ and ${\rm Re}\,z \ge  0$,
with $\varphi$ being  a  Bernstein function such that
\[
\varphi(0)=0,\qquad \varphi'(0)=1,\qquad|\varphi''(0)|<\infty.
\]
In particular, Euler's approximation
corresponds to
$\varphi(z)=\log(1+z)$ in \eqref{AsB}, while  Yosida's formula arises from \eqref{AsB1}
with $\varphi(z)=z/(z+1)$.
Along this way, apart from proposing a general view on semigroup approximation,
we equipped the approximation formulas with optimal  convergence rates depending on a ``generating function'' $\varphi,$
and covered in this way a number of known partial cases.
For more details of these developments we refer to \cite{GTJFA}.

Note however that (\ref{AsB}) and (\ref{AsB1}) can be put into a yet more general
and apparently more natural form
\begin{equation}\label{AsB0}
g^n_t (tz/n)-e^{-tz}\to 0,\qquad n\to \infty,
\end{equation}
where $g_t$ is defined by either $g_t=e^{-\varphi}$
or $g_t= e^{-t \varphi(\cdot/t)}, t>0.$
Moreover, observe that
 each of the functions $e^{-t\varphi}, t>0,$ is  bounded and completely monotone, that is
 a Laplace transform  of a positive bounded Borel measure $\mu_t$ on $[0,\infty).$
Thus it is natural to study \eqref{AsB0} for a family $(g_t)_{t > 0}$
of  bounded completely monotone functions  satisfying
\begin{equation}\label{gnorm}
g_t(0)=1,\qquad g_t'(0)=-1,\qquad g_t''(0)<\infty, \qquad t > 0,
\end{equation}
and not necessarily having a semigroup property with respect to $t.$
(The assumption $t \in (0,\infty)$ is not important and can be replaced by e.g. $t \in (0,a], a >0.$)
The formula \eqref{AsB0}
is a starting point of the paper.
It can also be viewed as a kind of so-called Chernoff's product formula frequently used in the study of $C_0$-semigroups,  e.g.
in \cite{Chernoff}, \cite{Chernoff1}, \cite{Chorin},  and \cite{Pf93}.
Note that there is a vast literature on Chernoff type approximation formulas and their rates,
see e.g. \cite{CaZa01}, \cite{CZ}, and \cite{Zagrebnov}. Unfortunately, our methodology does not apply in this setting in view of its  ``noncommutative'' nature.

To justify the framework of this paper, observe that the class of bounded completely monotone functions is considerably larger than
the class $\mathcal E=\{e^{-\varphi}: \varphi \,\, \text{is Bernstein}\},$ basic for \cite{GTJFA}.
For example, it is clear that if (a continuous extension of) a bounded completely monotone function $g$ has zeros on the imaginary axis, then
 $g \not \in \mathcal E.$
 Another less evident condition for $g \not \in \mathcal E$ is
 that $g$ is a Laplace transform of a bounded positive measure on $[0,\infty)$ with compact support,
 see e.g \cite[Lemma 7.5 and Corollary 24.4]{Sato}.
In general, measures $\mu$ giving rise to functions from $\mathcal E$ are called \emph{infinitely divisible} since
according to \eqref{bernst} for any such $\mu$ there exists a weak star continuous convolution semigroup
of Borel measures $(\mu_t)_{t \ge 0}$ satisfying $\mu_1=\mu,$ so that $\mu=(\mu_{1/n})^n$ for every $n \in \N.$
For more properties of this class as well as for more examples of measures that are not infinitely-divisible we refer to
\cite[Chapter 5]{SchilSonVon2010} and \cite[Chapters 2,5 and 6]{Sato}.

Thus, relying on the functional calculi ideas
from \cite{GTJFA} and the
approximation relation (\ref{AsB0}),
we develop further a comprehensive approach to the approximation theory
for bounded $C_0$-semigroups on Banach spaces. The approach allows one to unify
most of known approximation formulas and
equip them with optimal convergence rates, thus putting many results
from \cite{GTJFA} into their final form. Moreover, it leads to
 second order approximation formulas with optimal rates, and  approximation results
for holomorphic semigroups with sharp constants.
The heart matter of our considerations is positivity of measures
involved in various estimates, so the class of bounded completely monotone functions (or, equivalently, of Laplace transforms of positive bounded
Borel measures) seems to provide natural limits for the techniques of this paper (and also of \cite{GTJFA}).

Positivity aspects were also crucial in Pfeifer's studies mentioned above.
To give a flavor of his results we need to digress into the probabilistic terminology.
 Let $(N(\tau))_{\tau \ge 0}$ be a family of non-negative integer-valued
random variables with $E(N(\tau)) = \tau \xi$ for some $\xi > 0$ such that
$\sigma^2(N(\tau)) = o(\tau^2)$ for $\tau \to \infty$, where $E$ and $\sigma^2$ stand for expectation and variance, respectively.
If  $Y$ be a non-negative random variable with $E(Y) = \gamma$ such
that $\sigma^2(Y)$ exists, then, by \cite{Pf85}, for any bounded $C_0$-semigroup $(e^{-tA})_{t \ge 0}$ on $X$ and any $x \in X,$
\begin{equation}\label{P1}
e^{-tA} x = \lim_{\tau \to \infty} \psi_{N(\tau)} (E[e^{-YA/\tau}]) x
\end{equation}
in the strong sense, where $\psi_{N(\tau)}=E(t^N(\tau)), t\ge 0,$ and $t=\xi \gamma.$  Moreover, for $x \in {\rm dom}(A)$  the rate of convergence
in \eqref{P1}
can be estimated in terms of $\sigma (Y), \sigma (N(\tau))$ and $\|A x\|,$ and the same is true for
$x \in {\rm dom}(A^2)$ modulo a replacement of $\|A x\|$ by $\|A^2 x\|.$
Similarly, if $N$ is a non-negative integer-valued random variable
with $E(N) = \xi$ and $Y \ge 0$ is a real-valued random variable with $E(Y) = \gamma,$
then for $t=\xi \gamma$ one has
\begin{equation}\label{P2}
e^{-tA}x=\lim_{n \to \infty} \{\psi_N (E[e^{-AY/n}])\}^nx,
\end{equation}
again with the convergence rate estimates.
Specifying  $N$ in \eqref{P1} and \eqref{P2} (and similar asymptotic relations in \cite{Pf85}), one gets a number of approximation formulas with rates
that appear to be sharp by \cite{Fang}. For instance, choosing $N$ to be a binomial distribution in \eqref{P2}, one gets
so-called Kendall's approximation formula:
\[
\left \|\left((1-t)I + t e^{-A/n}\right)^n x - e^{-tA}x \right\|\le \frac{t(1-t)M}{2n} \|A^2 x\|
\]
for all $x   \in {\rm dom}(A^2), t \in (0,1),$ and integers $n >(t(1-t))^{-1}-6.$
For more details and explanations on probabilistic methodology we refer
to Pfeifer's papers and also to \cite{Butzer} and \cite{Chung}.
It is curious to note that in  \cite{Pf93}  Pfeifer gives an operator-probabilistic counterpart of the ``product formula''
\eqref{AsB0}. However he did not develop that result any further.
All positivity aspects are somewhat hidden in his arguments,
but they become transparent when
the expectations and other relevant quantities are written explicitly in terms of the distribution functions.

While the probabilistic approach
allowed one to
encompass a number of approximation formulas similarly to our studies,
it has also certain shortcomings when compared to functional calculi machinery.
In particular, it does not take into account the specifics of holomorphic semigroups.
Moreover, given a bounded complete function $g,$ functional calculus methods lead to explicit approximation formulas with convergence rates in terms of $g$ and its  derivatives at zero.
The probabilistic arguments are  somewhat less transparent since
they depend on a choice of auxiliary parameters (e.g. the distribution of $N$ as above),
thus requiring certain invention.
Finally,
 our framework includes the approximations $(g^n_t(\cdot/n))_{n \ge 1}$ when $g_t''(0)$ may not be finite, and the latter falls outside the scope
of limit theorems  in \cite{Pf85}  depending on  finite variances and thus assuming implicitly  $g_t''(0)< \infty.$
One should note that in fact one can often read off  approximations such as \eqref{AsB0} from probabilistic formulas similar to
\eqref{P1} and \eqref{P2}.
However that fact apparently escaped the attention of experts,
and the emphasis has been put on purely probabilistic techniques.
Overall, from a practical point of view, the two approaches can probably be considered as complementary rather than
covering each other.

Let us now review our main results. First we present  first order approximation estimates with rates,
comprising a number of known approximation schemes for semigroups into a single formula
with an optimal convergence rate. It is remarkable that the estimates are given in explicit a priori terms.
 The next result is a generalization of
\cite[Theorem 1.3]{GTJFA} to the setting of bounded completely monotone functions.

\begin{thm}\label{1Th0I}
Let $-A$ be the generator of a bounded $C_0$-semigroup $(e^{-tA})_{t\ge 0}$ on $X$, and let
 $(g_t)_{t >0}$ be a family of completely monotone functions such that for every $t> 0,$
\begin{equation}\label{ggg}
g_t(0)=1,\qquad g_t'(0)=-1, \quad \text{and} \quad g_t''(0)<\infty.
\end{equation}
Let  $M:=\sup\limits_{t\ge 0} \|e^{-tA}\|$.
Then for all  $n \in \mathbb{N}$ and $t>0,$
\begin{equation*}\label{RowA00I}
\| g_t^n(tA/n) x-e^{-tA}x\| \le M
\frac{(g_t''(0)-1 )}{2}\frac{t^2}{n} \|A^2 x\|,\qquad
x\in \dom(A^2),
\end{equation*}
\begin{equation*}\label{RowA001I}
\|g_t^n(tA/n) x-e^{-tA}x\| \le M
(g_t''(0)-1 )^{1/2}\frac{t}{\sqrt{n}} \|A x\|,\qquad
x\in \dom(A),
\end{equation*}
and, if $\alpha \in (0,2),$
\begin{equation*}\label{RowA011I}
\| g_t^n(tA/n) x-e^{-tA}x\| \le 4M \left((g_t''(0)-1)
\frac{t^2}{n}\right)^\frac{\alpha}{2} \|A^\alpha x\|,\qquad
x\in \dom(A^\alpha).
\end{equation*}
\end{thm}

Moreover, using the framework of completely monotone functions (and, implicitly, positivity of the corresponding measures),
we are able to obtain second order approximation formulas, again with optimal rates.
For this kind of results higher smoothness of the generating functions $g_t$ is required.

\begin{thm}\label{1Th+CI}
Let $-A$ be the generator of a bounded $C_0$-semigroup $(e^{-tA})_{t\ge 0}$ on $X$, and let
$(g_t)_{t > 0}$ be family of completely monotone functions such that, in addition to \eqref{ggg},
one has
\begin{equation*}
|g_t'''(0)|<\infty \qquad \text{and} \qquad g_t''''(0) <\infty,
\end{equation*}
for every $t > 0.$ Let  $M:=\sup\limits_{t\ge 0} \|e^{-tA}\|.$
Then for all  $t>0$, $n\in \N,$ and $x\in \dom(A^3),$
\begin{align*}\label{RowA0++I}
\|  g_t^n(tA/n)x-e^{-tA}x-(2n)^{-1}(g_t''(0)-1)& t^2e^{-tA}A^2x\|\\
\le& M C(g_t) t^3 n^{-3/2}\|A^3x\|,
\end{align*}
where
\[
C(g_t)=\left(\frac{(g_t''(0)-1)(g_t''''(0)-1)}{2}\right)^{1/2}.
\]
Moreover,
for all   $x\in \dom(A^4),$ $t>0$ and $n\in N,$
\begin{align*}
\|  g_t^n(tA/n) x-e^{-tA}x-(2n)^{-1}&{(g_t''(0)-1)} t^2e^{-tA}A^2x\|\\
\le&
M C_1(g_t)n^{-2}t^3(\|A^3x\|
+t\|A^4x\|),
\end{align*}
where
\[
C_1(g_t)=g_t''''(0)-1.
\]
\end{thm}
Theorems \ref{1Th0I} and \ref{1Th+CI} should be compared to Theorems $4.1$ and $4.4$ in \cite{Pf85} where similar results were obtained
using  a different, probabilistic language.

It is of interest for applications to obtain sharper versions of the above results if the semigroup
$(e^{-tA})_{t\ge 0}$ is holomorphic. In fact, it is this particular class of semigroups that attracted  most of attention
in the literature on semigroup approximation. By adjusting our functional calculus arguments accordingly,
we obtain analogues of Theorems \ref{1Th0I} and \ref{1Th+CI} for sectorially bounded holomorphic semigroups with improved rates (and constants)
in this particular case.
Here we formulate only one of these results and refer to Section \ref{holomorph} for other related statements.

\begin{thm}\label{tD3corI}
Let
$(g_t)_{t > 0}$ be as in Theorem \ref{1Th0I},
 let $-A$ be the generator of a sectorially bounded holomorphic $C_0$-semigroup $(e^{-tA})_{t \ge 0}$ on $X$,
and let for every $\beta \ge 0,$ $$M_\beta:=\sup_{t > 0} \|t^\beta A^\beta e^{-tA}\|.$$
Then for all $n\in \N$, and $t>0,$
\begin{equation*}\label{BB10ycI}
\|g_t^n(tA/n)x-e^{-tA}x\|
\le (2M_0+3M_1/2)\frac{(g_t''(0)-1)}{n}t\|Ax\|, \qquad x\in \dom(A).
\end{equation*}
Moreover, for all $n\in \N$ and $t>0,$
\begin{equation*}\label{BB30ycI}
\|g_t^n(tA/n)-e^{-tA}\|\le K \frac{(g_t''(0)-1)}{n},
\end{equation*}
and, if $\alpha \in (0,1),$
\[
\|g_t^n(tA/n)x-e^{-t A}x\|
\le 4K
\frac{(g_t''(0)-1)}{n}t^\alpha\|A^\alpha x\|, \quad x\in \dom(A^\alpha),
\]
where $K=3M_0+3M_1+M_2/2.$
\end{thm}

As an illustration of our technique we revisit Euler's approximation formula,
and show in  Theorem \ref{EulerI} below that the functional calculus estimates yield essentially the best constants there
(matching the corresponding numerical bounds !). This is a nice illustration of the
power of functional calculus.

\begin{thm}\label{EulerI}
Let $-A$ be the generator of a sectorially bounded holomorphic
$C_0$-semigroup $(e^{-tA})_{t\ge 0}$ on $X.$
 Then for all $\alpha\in [0,1]$, $t>0$, and $n\in \N,$
\begin{equation}\label{b10EI}
\left\|\left(1+\frac{t}{n}A\right)^{-n}x - e^{-tA}x \right\|\le
M_{2-\alpha} r_{\alpha, n} t^\alpha\|A^\alpha x\|, \quad x\in \dom(A^\alpha),
\end{equation}
where
\[
r_{\alpha,n}=\frac{1}{2n}+\frac{1-2\alpha}{12n^2},\quad \alpha\in [0,1/2),\quad
r_{\alpha,n}=\frac{1}{2n},\quad \alpha\in [1/2,1],\quad n\in \N.
\]
Moreover, there exists an absolute constant $C>0$ such that
for all $\alpha\in [0,1]$, $t>0$, and $n\in\N,$
\begin{align} \label{cggA0I}
&\left\|\left(1+\frac{t}{n}A\right)^{-n}x - e^{-tA}x -\frac{t^2A^2x}{2n}e^{-tA}x \right\| \\
\le&
C[M_{3-\alpha}
+M_{4-\alpha}]\frac{t^\alpha}{n^2} \|A^\alpha x\|, \qquad x\in \dom(A^\alpha).\notag
\end{align}
\end{thm}

Recall that the estimates of the form
\begin{equation}\label{ABCI}
\left\|\left(1+\frac{t}{n}A\right)^{-n}x - e^{-tA}x \right\|\le
C\frac{t^\alpha}{n} \|A^\alpha x\|,\qquad  x\in \dom(A^\alpha),
\end{equation}
where $n\in \N$ and $t>0$,
for some constant $C=C(A)$ and integer  values
$\alpha=0$ and $\alpha=1$ were obtained
as a partial case of more general rational approximations
for sectorially bounded holomorphic $C_0$-semigroups
in \cite[Theorems 4.2 and 4.4]{LTW91}.
See also \cite[Remark, p. 693-694]{BT79}.
A variant of (\ref{ABCI}) for all
$\alpha\in [0,1]$ and an explicit constant
$C=C(A)$ were also proved in \cite{GTJFA}
as a consequence of functional calculus approach.
In \cite[Theorem 1.3]{BP04},
 (\ref{ABCI}) was proved
for $\alpha=0$ and
$
C=4(\max(M_0,M_1))^3.
$
Moreover, (\ref{ABCI}) has been derived in \cite{ArZa} in the setting of Hilbert space $m$-accretive operators of angle $\alpha\in [0,\pi/2)$
 with
\[
C=\frac{K_\alpha}{\cos^2\alpha},\qquad \frac{\pi\sin\pi\alpha}{2\alpha}\le K_\alpha\le \min(\pi/\alpha-1,K_0),
\]
where $\pi/2\le K_0\le 2+2/\sqrt{3}$.
Note also that the higher order estimate  \eqref{cggA0I}
for $\alpha=0$ and with much worse constants  was obtained in \cite{Viel1}.
See also \cite{GTJFA} for additional references.

Finally, it will be convenient to fix the following notation for the rest of the paper.
For a closed linear operator $A$ on a complex Banach space $X$ we
denote by $\dom(A),$  $\ran(A)$,  and $\sigma(A)$  the
{\em domain}, the {\em range},   and the {\em spectrum} of $A$, respectively.
The norm-closure of the range is
written as $\cls{\ran}(A)$.
The space of bounded linear operators
on $X$ is denoted by $\Lin(X)$.
We let
\[
\C_{+}=\{z\in \C:\,{\rm Re}\,z>0\},\qquad\R_{+}=[0,\infty).
\]

\section{Preliminaries and further notation}\label{prelim}
Let us recall some function-theoretic facts
relevant for the following.
A nonnegative function $g\in \Ce^\infty\bigl(0,+\infty\bigr)$
is called {\em completely monotone} if
\[
 (-1)^n\frac{d^n g(t)}{dt^n}\geq 0, \qquad t >0
\]
for each $n\in\N$.  By Bernstein's theorem
\cite[Theorem 1.4]{SchilSonVon2010}
a function $g:(0,\infty)\to \mathbb R_+$
is completely monotone if and only if there
exists a (necessarily unique) positive Laplace-transformable Borel measure $\nu$ on $\R_+$
such that
\begin{equation}\label{CMLaplace}
g(t) = (\Lap \nu)(t): = \int\limits_0^\infty e^{-ts}\nu (ds) \quad \mbox{for all}\quad t > 0.
\end{equation}

The set of  completely monotone functions will be denoted by $\mathcal{CM}$,
and the set of  bounded completely monotone functions will be denoted by $\mathcal{BM}$.
If $g\in \mathcal{CM}$ has representation
\eqref{CMLaplace}, we will write $g\sim \nu$.
Each completely monotone function $g$ has close relative called a Bernstein function. Recall that a nonnegative function $f\in C^{\infty}(0,\infty)$ is Bernstein if  $f'$ is completely monotone.
An exhaustive treatment of completely monotone and Bernstein functions can be found in  \cite{SchilSonVon2010}.

For $g \in \mathcal{CM}$ and for $k \in \mathbb{N}\cup \{0\}$ we will denote \[
g^{(k)}(0):=\lim\limits_{z \to 0^+}g^{(k)}(z)
\]
where, in general, the limit can be infinite.

Let $\eM(\R_+)$ be the Banach algebra of bounded Borel measures on $\R_+.$
Note that if $\mu \in \eM(\R_+),$ then
the Laplace transform $\Lap\mu$ extends to a function holomorphic on $\mathbb C_+$ and continuous and bounded on $\cls{\C}_+$. Moreover,
 the space
\[ \Wip(\C_+) := \{ \Lap\mu : \mu \in \eM(\R_+)\}
\]
is a commutative Banach algebra with pointwise multiplication and with respect to the
norm
\begin{equation}\label{mmm}
\norm{\Lap \mu}_{\Wip} := \norm{\mu}_{\eM(\R_+)} = \abs{\mu}(\R_+),
\end{equation}
and the mapping
\[ \Lap : \eM(\R_+) \mapsto \Wip(\C_+)
\]
is an isometric isomorphism.

If $g\in \mathcal{BM}, g = \Lap \nu,$
then by Fatou's Lemma,
\begin{equation}\label{g0}
\infty > g(0)=\int_0^\infty \nu(ds)=\|g\|_{\Wip},
\end{equation}
due to the positivity of $\nu.$

If  $-A$ is the generator of a bounded $C_0$-semigroup $(e^{-tA})_{t\ge 0}$ on a Banach space $X$, then
 one defines
\begin{equation}\label{HFc}
g(A):=\int_0^\infty e^{-sA}\,\nu(ds),
\end{equation}
where the integral is understood as a strong Bochner integral.
The continuous algebra homomorphism
\begin{eqnarray*}
\Wip(\mathbb C_+) &\mapsto& \mathcal L(X)\\
g &\mapsto& g(A)
\end{eqnarray*}
is called the Hille-Phillips (HP-) functional calculus.
Clearly
$$
\|g(A)\|\le \|g\|_{\Wip} \sup_{t\ge 0}\|e^{-tA}\|,
$$
and  it is crucial that if $g \in \mathcal{BM},$ then by \eqref{g0} a much better estimate is available:
\[
\|g(A)\|\le g(0)\sup_{t\ge 0}\|e^{-tA}\|.
\]
The basic properties
of the Hille-Phillips functional calculus
can be found in  \cite[Chapter XV]{HilPhi}.
For a general approach to functional calculi, including the HP-calculus, see \cite{Haa2006}.

By a regularization procedure, the HP-calculus admits an extension to a class of functions much larger than
$\Wip(\C_+)$.  More precisely,
if $f: \C_+ \to \C$ is holomorphic such that there exists  $e\in
\Wip(\C_+)$ with $ef \in \Wip(\C_+)$ and the operator $e(A)$ is
injective, then one sets
\begin{eqnarray*}
\dom (f(A))&:=&\{x \in X: (ef)(A)x \in \ran(e(A)) \}\\
 f(A) &:=& e(A)^{-1} \, (ef)(A).
\end{eqnarray*}
In this case $f$ is called  regularizable, and $e$ is called
a regularizer for $f$. Such a definition of $f(A)$ does not
depend on the regularizer $e$ and $f(A)$ is a closed
 operator on $X$. Moreover, the set of all
regularizable functions $f$ forms an algebra ${\mathcal A}$ (depending on $A$).
The procedure gives rise to the mapping
\[ {\mathcal A} \ni f \tpfeil f(A)
\]
from
${\mathcal A}$ into the set of all closed operators on $X$ is
called the {\em extended Hille--Phillips {\rm (HP-)} functional
calculus} for $A$. The calculus possess a number of natural
properties, and we refer to \cite[Chapters 1 and 3]{Haa2006} for their exhaustive treatment.

In this paper, the next product rule will be important:
 if $f$ is regularizable and $g\in \Wip(\C_+)$,
then
\begin{equation}\label{hpfc.e.prod}
 g(A) f(A) \subseteq f(A) g(A) = (fg)(A),
\end{equation}
where products
of operators are considered on their natural domains. If $f \in \Wip (\mathbb C_+),$ then $f(A)$ is bounded and clearly $g(A)f(A)=(fg)(A).$
The rule will be mostly used for  $f(z)=z^\alpha, \alpha > 0,$ that are regularizable
(e.g. by a regularizer $e(z)=(z+1)^{-\alpha}$) and
thus belong to the extended HP-calculus, see e.g. \cite[p. 3060]{GTJFA} and \cite[Chapter 3]{Haa2006}.
Remark that
fractional powers $A^\alpha$ defined in this way coincide with the fractional
powers defined by means of the widely used extended holomorphic functional
calculus developed e.g. in \cite[Section 3]{Haa2006}. We omit the
details and refer to \cite{BaGoTa}, where, in particular,
compatibility of various calculi is thoroughly discussed.
(See also \cite[Proposition 1.2.7 and Section 3.3]{Haa2006}.)
Thus  several properties of fractional powers, well-known within the holomorphic functional calculus,
 will be used here without a special reference.

The following spectral mapping theorem for
the HP functional calculus will also be crucial. See e.g.
\cite[Theorem 16.3.5]{HilPhi} or \cite[Theorem 2.2]{GoHaTo12} for
its proof.
\begin{thm}\label{spmapping}
Let $g \in \Wip (\mathbb C_+)$ and  $-A$ be the generator of a bounded $C_0$-semigroup on a Banach space $X.$
Then
\begin{equation*}
\{g (\lambda): \lambda \in \sigma(A)\} \subset \sigma(g(A)).
\end{equation*}
\end{thm}

Since holomorphic semigroups will receive a special attention in our studies
it would be helpful to recall some of their basic properties.
If a $C_0$-semigroup $(e^{-tA})_{t \ge 0}$ on $X$
extends holomorphically to a sector
$\Sigma_\theta=\{\lambda \in \mathbb C: |\arg \lambda| <\theta\}$ for
some $\theta \in (0,\frac{\pi}{2}]$ and  $e^{-\cdot A}$ is
bounded and strongly continuous
in $\overline{\Sigma}_\xi$ whenever $0<\xi<\theta$, then
$(e^{-tA})_{t \ge 0}$ is said to be a \emph{sectorially bounded holomorphic} $C_0$-semigroup (of angle $\theta$).
It is well-known that sectorially bounded holomorphic semigroups
can be characterized in terms of their behavior on the real
axis. Namely,  $-A$ is the generator of a sectorially bounded
holomorphic $C_0$-semigroup $(e^{-tA})_{t \ge 0}$ on  $X$ if
and only if $e^{-tA}(X)\subset \dom (A), t>0,$ and
\[
\sup_{t\ge 0}\,\|e^{-tA}\|<\infty \qquad\mbox{and}\qquad
\sup_{t>0}\,\|tAe^{-tA}\|<\infty,
\]
see e.g. \cite[Theorem 4.6]{EngNag}.
Moreover, if $(e^{-tA})_{t \ge 0}$ is a sectorially bounded holomorphic $C_0$-semigroup, then for any $\beta >0$ one has
\begin{equation}\label{boundbeta}
\sup_{t >0 }  \| t^\beta A^\beta e^{-tA} \| <\infty,
\end{equation}
see e.g. \cite[Proposition 3.4.3]{Haa2006}.
The theory of $C_0$-semigroups, especially holomorphic ones, could essentially be put
into the framework of functional calculi. However, we preferred
to follow a more conventional treatment here.

\section{Some classes of bounded complete monotone functions and their estimates}\label{functions}
To be able to ``insert'' an operator into relations such as \eqref{AsB0} and to obtain
norm bounds for the corresponding operator expressions,
we will need a number of  estimates for completely monotone  functions.
Then the HP-calculus converts them into semigroup approximation formulas with rates.
Thus this section provides a function-theoretic background for subsequent functional calculi arguments.

\subsection{Estimates of functions for first order approximations}

To realize a functional calculus approach, we start with introducing a subclass $\mathcal{B}_1$
of normalized bounded completely monotone functions:
\begin{equation}\label{Gg}
\mathcal{B}_1:=\{g\in \mathcal{BM}:\;g(0)=1,\;g'(0)=-1\}.
\end{equation}
From the definition it follows that if $g \in \mathcal{B}_1, g \sim \nu,$ then
\begin{equation}\label{Rel1}
g(0)=\int_0^\infty \nu(ds)=1,\qquad
-g'(0)=\int_0^\infty s\,\nu(ds)=1,
\end{equation}
and for every $z>0,$
\begin{equation}\label{defB}
0\le g(z)\le 1,\qquad  -1\le g'(z)\le 0.
\end{equation}
The latter inequality implies that for each $g\in \mathcal{B}_1,$
\begin{equation}\label{defB0}
1-g(z)\le z, \qquad z>0.
\end{equation}
Given  $g\in \mathcal{B}_1$ define a sequence $(g_n)_{n \ge 1},$ one of the main
objects of our studies, by
\begin{equation}\label{Defg}
g_n(z):=g^n(z/n), \qquad z >0,
\end{equation}
and note that $(g_n)_{n \ge 1} \subset \mathcal B_1.$

For sharp norm estimates of operator functions,
the following subclass $\mathcal{B}_2$ of $\mathcal B_1$ will be useful:
\[
\mathcal{B}_2:=\{
g\in \mathcal{B}_1:\,g''(0)<\infty\}.
\]
Note that, if $g\in \mathcal{B}_2, g \sim \nu,$ then
\begin{equation}\label{secD}
\int_0^\infty (s-1)^2\,\nu(ds)=g''(0)-1,
\end{equation}
hence $g''(0)\ge 1$ and $g''(0)=1$ if and only if
$g(z)=e^{-z}$. Clearly, if  $g\in \mathcal{B}_2$
then  $(g_n)_{n \ge 1} \subset \mathcal{B}_2.$ Moreover,
\begin{equation}\label{h00}
g_n''(0) = 1+ \frac{g''(0)-1}{n}.
\end{equation}

The next statement is similar  to
\cite[Proposition 4.4]{GTJFA}, where we considered
completely monotone functions $g$ of the form $g=e^{-\varphi}$, with a normalized
Bernstein function $\varphi.$ It is fundamental in relating the semigroup approximation
to the HP-calculus.
While a  result more general than \eqref{GtSDop} can be found in \cite[Theorem C]{Rosser} (the reference,  which was not known
to the authors of \cite{GTJFA} at the time of writing that paper), the
statement below suffices for our purposes, and we provide it with an independent simple proof.
For  $g \in \mathcal{CM}$ and $\alpha\in [0,2]$  define
\[
\Delta_\alpha^{g}(z):=\frac{g(z)-e^{-z}}{z^\alpha},\qquad  z>0.
\]
\begin{lemma}\label{T1Dop}
Let $g\in \mathcal{B}_1,$ $g \sim \nu.$
Then
$\Delta_2^g \in \mathcal {CM}.$
Moreover
\begin{equation}\label{GtSDop}
\Delta_2^g (z)=\int_0^\infty e^{-z s}G(s)\,ds,\qquad z >0,
\end{equation}
where $G$ is a function  continuous on $\mathbb R_+$ and given by
\begin{align}\label{dopp}
G(s)=
\begin{cases} \int_0^s (s-\tau)\, \nu(d\tau),& \qquad s\in [0,1],\\
\int_s^\infty (\tau -s)\,\nu(d\tau),& \qquad s>1.
\end{cases}
\end{align}
If $g\in \mathcal{B}_2$, then
\begin{equation}\label{cor12}
\|\Delta_2^g\|_{\Wip(\C_+)}=\Delta_2^g(0)=\frac{g''(0)-1}{2}.
\end{equation}
\end{lemma}

\begin{proof}

First we prove the representation \eqref{GtSDop}.
Taking into account \eqref{CMLaplace},
we have for $z>0:$
\begin{eqnarray*}
\Delta_2^g(z)=z^{-2} \Delta_0^g(z)&=&\int_0^\infty e^{-z\tau} \tau d\tau \int_0^\infty e^{-zs}
\,\nu(ds) -\int_1^\infty e^{-zs}(s-1)\,ds
\\
&=&\int_0^\infty e^{-zs} \,\int_0^s (s-\tau) \,\nu(d\tau) \, ds-
\int_1^\infty e^{-zs}(s-1)\,ds\\
&=&\int_{0}^{\infty}e^{-zs} H(s)\, ds,
\end{eqnarray*}
where
\[
 H(s)= \int_0^s (s-\tau)\, \nu(d\tau)-
(s-1) 1_{[0,\infty)}(s-1).
\]

It is clear that if $ s \le 1$ then $H(s)=G(s).$ We prove that
$H(s)=G(s)$ for $s > 1$ as well. If $s >1,$ then using
\eqref{Rel1}, we infer that
\begin{eqnarray*}
H(s)&=&s\left(\int_0^s\, \nu(d\tau)-1\right)+1-\int_0^s \tau \,\nu(d\tau)\\
&=&-s\int_{s+}^\infty \,\nu(d\tau)+
\int_{s+}^\infty \tau \,\nu(d\tau)\\
&=&\int_s^\infty (\tau -s)\,\nu(d\tau),
\end{eqnarray*}
hence  $H(s)=G(s).$ Thus \eqref{GtSDop} holds. Since
$G(s)>0,$ $s>0,$ the function $\Delta_2^g$ is completely monotone.
The continuity of $G$ at $s=1$ follows from the identity
(see (\ref{Rel1}))
\[
\int_0^\infty \nu(d\tau)=\int_0^\infty \tau\nu(d\tau),
\]
so that
\[
G(1)=\int_0^1(1-\tau)\,\nu(d\tau)
=\int_1^\infty (\tau-1)\,\nu(d\tau)=\lim_{s\to 1+}\,G(s),
\]
while the continuity of $G$ at other points of $[0,\infty)$ is obvious.

To prove (\ref{cor12}), taking into account $\Delta_2^g \in \mathcal{CM}$ and
applying  Fatou's lemma and de l'Hopital's rule twice, we obtain:
\begin{align*}
\|\Delta_2^g\|_{\Wip(\C_+)}=&\lim_{z\to 0+}\,
\frac{g(z)-e^{-z}}{z^2}
=\frac{g''(0)-1}{2}.
\end{align*}
\end{proof}

It is useful to note several elementary properties of the density $G$
defined above.
In particular, $G$ is non-decreasing on $[0,1]$ and non-increasing on
$s\ge 1$. Moreover, $G(0)=0,$
\begin{equation}\label{Gless2}
\max\{G(s):s>0\}=G(1)
=\int_0^1(1-\tau)\,\nu(dt)\le 1,
\end{equation}
and
\begin{equation}\label{Gless3}
\lim_{s\to\infty}\,G(s)\le \lim_{s\to\infty}\,\int_s^\infty \tau\,\nu(d\tau)=0.
\end{equation}
The latter property  implies
\begin{equation}\label{rem}
\lim_{\delta\to +0}\, \delta \int_0^\infty e^{-\delta s} G(s)\,ds=0.
\end{equation}
by the regularity of Abel's summation.
Finally, if $g \in \mathcal B_2,$ then  $\Delta_2^g \in L^1(0,\infty),$ so
the corresponding $G$ satisfies
\[
\int_{0}^{\infty}\frac{G(s)\, ds}{s}<\infty
\]
by Fubini's theorem.

\begin{remark}\label{bounds}
Lemma \ref{T1Dop} yields several simple but useful bounds for $g-e^{-z}, g\in \mathcal{B}_1.$
 Since $g(z)\le 1$ for $z \ge 0,$ from Lemma \ref{T1Dop} it
follows that
\begin{equation}\label{pos1}
0\le g(z)-e^{-z}\le 1,\qquad z\ge 0.
\end{equation}
Since $g'(z)\le 0$, $z\ge 0$, we also have
\begin{equation}\label{pos2}
g(z)-e^{-z}=\int_0^z (g'(s)+e^{-s})\,ds\le
z,\qquad z\ge 0.
\end{equation}
Furthermore, from (\ref{pos1}) and (\ref{pos2}), we conclude that
\begin{equation}\label{pos3}
0\le g(z)-e^{-z}\le \min\,\{1,z\}
\le \frac{2z}{z+1},\qquad z\ge 0.
\end{equation}
\end{remark}

The next proposition shows that the sequence $(g_n)_{n \ge 1}$ defined by \eqref{Defg} may serve as an
approximation of the exponent  and justifies  our
functional calculus approach to the semigroup approximation.

\begin{prop}\label{exp}
Let $g\in \mathcal{B}_1$.
Then for all $t>0$ and $n\in \N,$
\begin{equation}\label{LB}
0\le g_n(t)-e^{-t}\le
t(1+g'(t/n)),
\end{equation}
and, in particular,
\begin{equation}\label{LBB}
\lim_{n\to \infty}\,g_n(t)=e^{-t},
\end{equation}
uniformly in $t$ from compact subsets of $[0,\infty).$
\end{prop}

\begin{proof}
Consider the quotients $g_n(t)/e^{-t}, n \in \mathbb N.$
Using \eqref{defB}, \eqref{defB0} and
the monotonicity of $1+g'$ and $g$, we have
\begin{align*}
\log\frac{g^n(t/n)}{e^{-t}}=&
n\left(\log g(t/n) + t/n \right)
=n\int_0^{t/n}\frac{g(s)+g'(s)}{g(s)}\,ds\\
\le& n\int_0^{t/n}\frac{1+g'(s)}{g(s)}\,ds
 \le \frac{t}{g(t/n)}(1+g'(t/n)).
\end{align*}
Then, since
\[
t-s\le t\log(t/s),\qquad 0<s<t,
\]
we obtain
\[
0\le g^n(t/n)-e^{-t}\le
t (1+g'(t/n)).
\]
\end{proof}

\begin{rem}\label{B22}
If $g\in \mathcal{B}_2$, then
the inequalities
$1+g'(t)\le g''(0)t,$ $t\ge 0,$
and \eqref{LB}
imply that for all $n \in \N$ and $t \ge 0,$
\[
0\le g_n(t)-e^{-t}\le
g''(0)\frac{t^2}{n}.
\]
Moreover, by \eqref{cor12},
we have a slightly stronger inequality:
\begin{equation}\label{LB0}
0\le g_n(t)-e^{-t}\le
(g''(0)-1)\frac{t^2}{2n}, \qquad n \in \N, \quad t \ge 0.
\end{equation}
\end{rem}
We proceed with an auxiliary statement similar to Lemma \ref{T1Dop}.
While Lemma \ref{T1Dop} will be useful in the study of approximations for general $C_0$-semigroups,
Lemma \ref{LB1} below will play the same role for approximations of holomorphic
$C_0$-semigroups.
\begin{lemma}\label{LB1}
For every $g\in \mathcal{B}_1$ let $s_g(z):=z^{-1}({g(z)+g'(z)}),\, z>0.$
Then
$s_g \in \mathcal{CM}.$
If moreover  $g\in \mathcal{B}_2,$ then $s_g \in \Wip(\C_+)$ and
\begin{equation}\label{LemA10}
\|s_g\|_{\Wip(\C_+)}
=g''(0)-1.
\end{equation}
\end{lemma}

\begin{proof}
The fact that $s_g\in \mathcal{CM}$ for $g\in \mathcal{B}_1$ follows from \cite[Theorem A]{Rosser}.
Again in our particular case, the argument is quite simple and we give it below.
Let $g \sim \nu.$
Using Fubini's theorem, we have
\begin{align*}
\frac{g(z)+g'(z)}{z}
=&
\int_0^\infty e^{-zs} \int_0^s\nu(d\tau)\,ds
-\int_0^\infty e^{-zs} \int_0^s\tau \nu(d\tau)\,ds\\
=&\int_0^\infty e^{-zs} G_0(s)\,ds,\quad z>0,
\end{align*}
where
\[
G_0(s)=\int_0^s (1-\tau)\,\nu(d\tau),\quad s\in [0,1],
\]
and, by \eqref{Rel1},
\[
G_0(s)=
\int_{s}^\infty (\tau-1)\,\nu(d\tau)\ge 0,\qquad s>1.
\]
Thus $G_0$ is  positive on $[0,\infty)$ and $s_g \in \mathcal{CM}$ by Bernstein's theorem.
If $g\in \mathcal{B}_2$, then
(\ref{LemA10}) follows
from Fatou's Lemma and L'H\^ opital's rule.
\end{proof}

As a straightforward implication of Lemma \ref{LB1} observe that
for $g\in \mathcal{B}_1$ one has
\[
g(z)+g'(z)\ge 0,\qquad z>0.
\]
If moreover $g\in \mathcal{B}_2$, then
\begin{equation*}\label{der2}
g(z)+g'(z)\le (g''(0)-1)z,\qquad z\ge 0.
\end{equation*}

Let us now introduce a functional $L$ on $\mathcal {BM}$ crucial for operator-norm estimates via the HP-calculus.
For $g \in \mathcal{BM},$ $g \sim \nu,$ define
\begin{equation}\label{functional}
L[g]:=\int_0^1 (1-s)\,\nu(ds).
\end{equation}
The following statement clarifies the meaning of $L.$

\begin{prop}\label{CMcor}
If $g\in \mathcal{B}_1$, $g\sim \nu$, then
$\Delta^g_1\in \Wip(\C_+)$ and
\begin{equation}\label{LemA2}
\|\Delta^g_1\|_{\Wip(\C_+)}=2 L[g].
\end{equation}

If $g\in \mathcal{B}_2$, then
\begin{equation}\label{cor120}
\|\Delta_1^g\|_{\Wip(\C_+)}\le (g''(0)-1)^{1/2},
\end{equation}
and
\begin{equation}\label{cor122}
\|\Delta_1^{g_n}\|_{\Wip(\C_+)}\le \left(\frac{g''(0)-1}{n}\right)^{1/2}.
\end{equation}
\end{prop}

\begin{proof}
Note that
\begin{align*}
\Delta^g_1(z)=z^{-1} \Delta_0^g(z)
=\int_0^\infty e^{-zs} \,\int_0^s  \,\nu(d\tau) \, ds-
\int_1^\infty e^{-zs}\,ds.
\end{align*}
Hence
\begin{align*}
\|\Delta^g_1\|_{\Wip(\C_+)}=&
\int_0^1 \int_0^s  \,\nu(d\tau) \, ds+
\int_1^\infty (1-\int_0^s\nu(d\tau))\,ds\\
=&\int_0^\infty |1-\tau|\,\nu(d\tau)
=2\int_0^1 (1-\tau)\,\nu(d\tau),
\end{align*}
i.e. \eqref{LemA2} holds.

In view of \eqref{secD} and \eqref{Rel1}, the estimate (\ref{cor120})
follows from
 \eqref{LemA2}
by Cauchy's inequality.
\end{proof}

To get the convergence rates on the domains of fractional powers of (negative) semigroup generators,
we will need the next interpolation result which is a  partial case of
\cite[Lemma 4.1]{GTJFA}.

\begin{lemma}\label{C+global}
Let $F \in A^1_+(\mathbb{C}_+)$ be such that $z^{-1}F \in A^1_+(\mathbb{C}_+)$ and
$$
\| F \|_{A^1_+(\mathbb{C}_+)} = a \quad\mbox{and}
 \quad \|z^{-1} F \|_{A^1_+(\mathbb{C}_+)} =b.
 $$
Then for every $\alpha \in [0,1]$ one has $z^{-\alpha} F \in A^1_+(\mathbb{C}_+)$ and
\begin{equation}\label{314}
\| z^{-\alpha} F \|_{A^1_+(\mathbb{C}_+)} \leq 2^{1+\alpha}  a^{1-\alpha} b^{\alpha}.
\end{equation}
\end{lemma}
Combining now Lemma \ref{C+global} with Proposition \ref{CMcor}
we estimate $\Delta_\alpha^g$ with $\alpha$ from either $[0,1]$ or $[0,2]$  depending
on the smoothness of $g$ at zero.

\begin{cor}\label{CintC}
If  $g \in \mathcal{B}_1$, then for every $\alpha \in [0,1],$
\begin{equation}\label{Int11}
\left\| \Delta_\alpha^g\right\|_{A^1_+(\mathbb{C}_+)}
\le 8 (L[g])^\alpha.
\end{equation}
If  moreover $g \in \mathcal{B}_2$, then for every $\alpha \in [0,2],$
\begin{equation}\label{Int12}
\left\| \Delta_\alpha^g\right\|_{A^1_+(\mathbb{C}_+)}
\le 4 \left(g''(0)-1\right)^\frac{\alpha}{2}.
\end{equation}
\end{cor}

\begin{proof}
First, let $\alpha \in [0,1]$ be fixed and $g \in \mathcal B_1.$ Observe that
\begin{equation}\label{note1}
\left\| \Delta_0^g \right\|_{A^1_+(\mathbb{C}_+)} \le \left\| g\right\|_{A^1_+(\mathbb{C}_+)} + \left\| e ^{-z} \right\|_{A^1_+(\mathbb{C}_+)} =2,
\end{equation}
and $\|\Delta^g_1\|_{\Wip(\C_+)}=2 L[g]$ by Corollary \ref{CMcor}.
Then, using   Lemma \ref{C+global} with
  $F=\Delta_0^g$, we obtain
  that
\[
\left\| \Delta_\alpha^g  \right\|_{A^1_+(\mathbb{C}_+)}
\le 4\|\Delta_1^g\|^\alpha_{A^1_+(\mathbb{C}_+)}.
\]
Thus, for $\alpha \in [0,1],$
\eqref{LemA2}  implies \eqref{Int11},  and  \eqref{Int12} follows from \eqref{cor120}.

Next, if  $\alpha\in (1,2]$ and  $g \in \mathcal{B}_2$, then
 \eqref{cor120}, \eqref{cor12},
 and  Lemma \ref{C+global} with
  $F=\Delta^g_1$
 yield
 \begin{align*}
\left\| \Delta_\alpha^g  \right\|_{A^1_+(\mathbb{C}_+)}
=&\left\|z^{-\alpha+1} \Delta_1^g  \right\|_{A^1_+(\mathbb{C}_+)}\\
\le& 2^{ \alpha}
\left(g''(0)-1\right)^{(2-\alpha)/2}
\left(\frac{g''(0)-1}{2}\right)^{\alpha-1}\\
= & 2 \left(g''(0)-1\right)^\frac{\alpha}{2}.
\end{align*}
\end{proof}

The estimate in \eqref{Int11} looks rather implicit. So we proceed with giving explicit bounds for  $L[g]$ when $g \in \mathcal B_1.$

\begin{lemma}\label{Lem22}
For every $g\in \mathcal{B}_1$ one has
\begin{equation}\label{22est}
L[g]\le
\left(
\frac{(1+g'(1))\int_0^1 g(s)\,ds}{(1-g(1))^2}-1\right)^{1/2}
+2e\int_0^1 (g(s)+g'(s))\,ds.
\end{equation}
\end{lemma}

\begin{proof}
Let $g\in\mathcal{B}_1$, $g\sim\nu$.
Set
\[\beta=\int_0^1 g(s)\,ds,\qquad
\gamma=1-g(1),
\]
and define
\[
\tilde{g}(z)=\frac{1}{\beta}
\int_{\beta/\gamma z}^{\beta/\gamma z+1} g(s)\,ds, \qquad z >0.
\]
We will estimate $L[g]$ in terms of $L[\tilde g],\beta$ and $\gamma.$
Then the bound \eqref{22est} will follow easily.

Note that
$
0< \gamma \le \beta\le 1,
$
and
\begin{equation}\label{Ddd}
\tilde{g}\in \mathcal{B}_2,\qquad
\tilde{g}''(0)=\frac{\beta(1+g'(1))}{\gamma^2}.
\end{equation}
Moreover, we have
\[
\tilde{g}(z)=\frac{1}{\beta}
\int_0^\infty e^{-(\beta/\gamma)  z s}
\frac{(1-e^{-s})}{s}\,\nu(ds)
=\int_0^\infty e^{-z s}\,
\tilde{\nu}(ds), \qquad z >0,
\]
where
\[
\tilde{\nu}(ds)=\frac{(1-e^{-(\gamma/\beta)s})}{\gamma s}\,
\nu((\gamma/\beta)ds).
\]

Since $\beta\le 1$,
\begin{align*}
L[\tilde{g}]=&
\int_0^1 (1-s)\,
\frac{(1-e^{-(\gamma/\beta)s })}{\gamma s}\,
\nu((\gamma/\beta)\,ds)\\
=&\int_0^{\gamma/\beta} \left(1-s\frac{\beta}{\gamma}\right)\,
\frac{(1-e^{-s })}{\beta s}\,
\nu(\,ds)\\
\ge& \int_0^{\gamma/\beta} \left(1-s\frac{\beta}{\gamma}\right)\,
\frac{(1-e^{-\beta s })}{\beta s}\,
\nu(\,ds),
\end{align*}
hence
\begin{align*}
L[g]-L[\tilde{g}]\le& \int_0^1 (1-s)\,\nu(ds)
-\int_0^{\gamma/\beta} \left(1-s\frac{\beta}{\gamma}\right)\,
\frac{(1-e^{-\beta s })}{\beta s}\,
\nu(ds)\\
=&\int_{\gamma/\beta}^1 (1-s)\,\nu(ds)
+\int_0^{\gamma/\beta}
Q(s)\,
\nu(ds),
\end{align*}
where
\begin{align*}
Q(s)=&(1-s)-\left(1-s\frac{\beta}{\gamma}\right)\,
\frac{(1-e^{-\beta s })}{\beta s}\\
=&\left(\frac{\beta}{\gamma}-1\right)
\frac{(1-e^{-\beta s })}{\beta}
+(1-s)\left\{1-
\frac{(1-e^{-\beta s })}{\beta s}\right\}.
\end{align*}
Observe now that for every $s\in [0,\gamma/\beta],$
\[
Q(s)\le  \left(\frac{\beta}{\gamma}-1\right)
s+(1-s)\frac{\beta s}{2}
\le \left(1-\frac{\gamma}{\beta}\right)
+\frac{(1-s)}{2}.
\]
Thus
\[
L[g]-L[\tilde{g}]
\le  \left(1-\frac{\gamma}{\beta}\right)+
\frac{1}{2}L[g],
\]
and then, in view of \eqref{secD} and \eqref{cor120},
\begin{align}\label{ELL}
L[g]\le& 2L[\tilde{g}]+ 2\left(1-\frac{\gamma}{\beta}\right)\\
\le& \left(\int_0^\infty (1-s)^2\,\tilde{\nu}(ds)\right)^{1/2}+
2\left(1-\frac{\gamma}{\beta}\right)\notag \\
=& (\tilde{g}''(0)-1)^{1/2}+
2\left(1-\frac{\gamma}{\beta}\right).\notag
\end{align}
Since, recalling the definition of $\gamma$ and $\beta,$
\begin{align*}\label{betag}
1-\frac{\gamma}{\beta}=&
\frac{\int_0^1 g(s)\,ds-1+g(1)}
{\int_0^1 g(s)\,ds}
\le \frac{\int_0^1 g(s)\,ds-1+g(1)}
{g(1)}\\
\le& e \int_0^1 (g(s)+g'(s))\,ds, \notag
\end{align*}
the statement
follows from
\eqref{ELL} and
\eqref{Ddd}.
\end{proof}

Now we are able to describe the convergence rate for $(g_n)_{n \ge 1}$ in \eqref{LBB}
by means of the boundary behavior of $g'$ at zero.

\begin{cor}\label{diffC}
Let $g\in \mathcal{B}_1$. Then for all $n \in \mathbb N,$
\begin{equation}\label{holA}
L[g_n]
\le 2e\left(1+\frac{1}{|g'(1/n)|}\right)
\sqrt{1+g'(1/n)}.
\end{equation}
\end{cor}

\begin{remark}
It is instructive to note that by \eqref{Gg} for each $g \in \mathcal B_1$ one has $\lim_{n \to \infty}(1+g'(1/n))=0.$
\end{remark}

\begin{proof}
First observe that for every  $n\in \N,$
\[
1-g_n(1)=\int_0^1 g^{n-1}(s/n)|g'(s/n)|\,ds
\ge e^{-1}|g'(1/n)|.
\]

Next, using the inequalities \eqref{LB} for $t=1$ and an obvious bound $g \le 1$, we obtain
\begin{align*}
&(1+g_n'(1/n))\int_0^1 g_n(s)\,ds-(1-g_n(1))^2\\
=&g^{n-1}(1/n)[g(1/n)+g'(1/n)]\int_0^1 g^n(s/n)\,ds\\
+&(1-g^n(1/n))
\left(\int_0^1 g^n(s/n)\,ds-(1-g^n(1/n))\right)\\
\le& [1+g'(1/n)]
+(1-e^{-1})
\left(
\int_0^1 s(1+g'(s/n))\,ds+
g^n(1/n))-e^{-1}\right)\\
\le& (1+3(1-e^{-1})/2)  (1+g'(1/n))\\
\le&
2 (1+g'(1/n)), \qquad n \in \N.
\end{align*}
So, by \eqref{22est}, for every $n \in \N,$
\[
L[g_n]\le
e\sqrt{2} \frac{(1+g'(1/n))^{1/2}}{|g'(1/n)|}
+2e(1+g'(1/n)),
\]
and \eqref{holA} follows.
\end{proof}
Thus if $g''(0)=\infty$ then the convergence rate in Corollary \ref{diffC} depends on the behavior of $g$
in a neighborhood of zero. This is in contrast to Proposition \ref{CMcor} where the assumption $g''(0) < \infty$ yields a rate estimate
independent of $g$ (up to a constant).
The rates coincide when the latter condition holds, and they are optimal in this case, see Proposition \ref{Opt1}.

We will use Corollary \ref{diffC} to characterize polynomial decay rates
for $\|\Delta_1^{g_n}\|_{A^1_+(\mathbb{C}_+)}.$ To this aim the next lemma
will be crucial.

\begin{lemma}\label{LemmaN}
Let $f\in \mathcal{CM}$, $f\sim \mu,$ and
\begin{equation}\label{IntFF}
\int_0^1 f(s)\,ds<\infty.
\end{equation}
If $\gamma\in (0,1)$, the  the following conditions are equivalent.
\begin{itemize}
\item [\emph{(i)}] There exists $c_1>0$ such that $ f(\tau)\le c_1\tau^{-1+\gamma},\qquad \tau\in (0,1).$
\item [\emph{(ii)}]  There exists $c_2>0$ such that $\int_0^s \mu(d\tau)\le c_2 s^{1-\gamma},\qquad s\ge 1.$
\item [\emph{(iii)}]  There exists $c_3>0$ such that $\int_0^{1/n} f(\tau)\,d\tau\le c_3n^{-\gamma},\qquad n\in \N.$
\end{itemize}
\end{lemma}

\begin{proof}
$(i)\Longrightarrow (ii)$:
Since for every  $s\ge 1,$
\[
\int_0^s\mu(d\tau)\le
e\int_0^\infty e^{-\tau/s}\,\mu(d\tau)=
ef(1/s)\le e c_1s^{1-\gamma},
\]
the implication follows.

$(ii)\Longrightarrow (iii):$
By \eqref{IntFF},
\[
\int_0^\infty \frac{\mu(d\tau)}{1+\tau}<\infty,
\]
therefore
\begin{equation}\label{Us1}
\lim_{s\to\infty}\,\frac{1}{s}\int_0^s \mu(d\tau)\le \lim_{s \to \infty} 2 \int_{0}^\infty \frac{(1+\tau)1_{[0,t]}(\tau)}{1+s}\frac{\mu(d \tau)}{1+\tau}=0,
\end{equation}
by the  bounded convergence theorem.
Moreover, by \eqref{pos3} we have
\[
\frac{1-e^{-s/n}}{s}\le \frac{2}{n+s},\qquad s>0,\quad n\in \N,
\]
hence, using (\ref{Us1}) and integrating by  parts, we obtain
for every $n\in \N$ and some $c_2>0$:
\begin{align*}
\int_0^{1/n} f(\tau)\,d\tau=&
\int_0^\infty \frac{1-e^{-s/n}}{s}\,\mu(ds)\le
2\int_0^\infty\frac{\mu(ds)}{n+s}\\
\le& 2\int_0^\infty \frac{\int_0^s \mu(d\tau) \, ds}{(n+s)^2}
\le c_2\int_0^\infty \frac{s^{1-\gamma}\,ds}{(n+s)^2}\\
=&c_2 \frac{\pi (1-\gamma)}{\sin(\pi\gamma)}n^{-\gamma}.
\end{align*}

$(iii)\Longrightarrow (i):$
If $t\in ((n+1)^{-1},n^{-1}), n \in \mathbb N$, then as  $f$ is non-increasing there is $c_3 >0$ such that for every $\tau\in (0,1],$
\begin{align*}
f(\tau)\le f((n+1)^{-1})\le(n+1)\int_0^{\frac{1}{n+1}}f(s)\,ds
\le
c_3 (n+1)^{1-\gamma}
 2c_3 \tau^{-1+\gamma}.
\end{align*}
\end{proof}

\begin{cor}\label{corNew}
Let $g \in \mathcal B_1, g \sim \nu.$
Then the following conditions are equivalent.
\begin{itemize}
\item [\emph{(i)}] There exists $c_1>0$ such that $g''(\tau)\le c_1\tau^{-1+\gamma},\quad \tau\in (0,1].$

\item [\emph{(ii)}] There exists $c_2>0$ such that $\int_0^s \tau^2\mu(d\tau)\le c_2 s^{1-\gamma},\quad s\ge 1.$

\item [\emph{(iii)}] There exists $c_3>0$ such that $\,  1+g'(1/n)\le c_3n^{-\gamma},\quad n\in \N.$
\end{itemize}
\end{cor}

\begin{proof}
The statement follows from Lemma \ref{LemmaN}
applied  to $f=g''.$ It suffices to note that  $g''\in\mathcal{CM},$  $g''\sim \tau^2\nu(d\tau),$ $g''$ is integrable on $[0,1],$ and
\[
1+g'(1/n)=\int_0^{1/n} g''(\tau)\,d\tau,\qquad n\in \N.
\]
\end{proof}

Observe that
if
\[
g(z):=(1-\gamma)+\gamma(\gamma+1)\int_0^\infty\frac{e^{-zs}\,ds}{(s+1)^{2+\gamma}},\qquad\gamma\in (0,1),
\]
then $g \sim \nu,$
\[
\nu(ds):=(1-\gamma)\delta_0+\gamma(\gamma+1)\frac{ds}{(s+1)^{2+\gamma}},
\]
so $g\in \mathcal{B}_1\setminus \mathcal{B}_2.$
On the other hand, by Corollary \ref{corNew},
 \[
1+g'(1/n)\le cn^{-\gamma},\qquad \gamma\in (0,1),
\]
for some $c>0.$

\subsection{Estimates of functions for higher order approximations and constants}

To treat higher order approximations for bounded $C_0$-semigroups and obtain sharp  estimates of constants in our approximation formulas,
we will need several classes of completely monotone functions finer than $\mathcal B_1$ and $\mathcal B_2$.
For $k=3,4,$ let
\begin{align*}
\mathcal{B}_k=&\{g\in \mathcal{B}_2:\;\; |g^{(k)}(0)|<\infty\},\\
\mathcal{B}_{k,\infty}=&\left\{g\in \mathcal{B}_k:\;\; \int_1^\infty \frac{g(s)}{s}\,ds<\infty\right\}.
\end{align*}
Note that if $g \in \mathcal B_k$ belongs also to $L^m (0,\infty), m \in \N,$ then $g \in \mathcal{B}_{k,\infty}.$ The classes $\mathcal B_k, k \in \N,$ have been mentioned briefly in \cite{Lorch}.

For $g \in \mathcal B_2$ let the function $G$ be given by Lemma \ref{T1Dop}.
For the formulations of our higher order approximation results, it will be convenient
to introduce the following functionals:
\begin{align}
\label{defab1} a[g]:=&\int_0^\infty G(s)\,ds, \qquad g \in \mathcal B_2,\\
\label{defab2} b[g]:=&\int_0^\infty (1-s)
 G(s)\, ds, \qquad g \in \mathcal B_3.
\end{align}
Note that for each $g \in \mathcal B_2,$
\begin{align}\label{defab3}
a[g] =
\Delta_2^g(0)=\frac{g''(0)-1}{2},
\end{align}
and, since for $g \in \mathcal B_3,$
\[
\lim_{z\to 0+}\,\frac{d}{dz} \Delta^g_2(z)=
\frac{g'''(0)+1}{6},
\]
we have
\[
b[g]=\Delta_2^g(0)+\lim_{z\to 0}
\frac{d}{dz}\Delta_2^g(z)
=\frac{3g''(0)+g'''(0)-2}{6}.
\]

In order to obtain optimal bounds for  higher order approximations of holomorphic semigroups, we will need yet another family of
 auxiliary functionals on $B_{2,\infty}.$

For $\alpha \in [0,1]$
and $g\in \mathcal{B}_{2,\infty}$ define
\begin{equation}\label{cgg}
c_\alpha[g]:=
\frac{1}{\Gamma(2-\alpha)}\int_0^\infty
\Delta_{1+\alpha}^g(z)\,dz,
\end{equation}
 where $\Gamma$ is the Gamma-function:
\[
\Gamma(z):=\int_0^\infty t^{z-1}e^{-t}\,dt,\qquad z>0.
\]
Note that for each $g\in \mathcal{B}_{2,\infty}$ one has
$c_\alpha[g]<\infty$, so that the mapping $c_\alpha: \mathcal{B}_{2,\infty}\mapsto [0,\infty)$ is well-defined.
We are interested in sharp estimates of $c_\alpha[g]$ in terms of $g,$ and several representations
for $c_\alpha$ given below will serve just that purpose.

The following proposition expresses $c_\alpha$ in terms of the function $G$ corresponding to $g$ by Lemma \ref{T1Dop}.

\begin{prop}\label{DeltaG}
If $g\in \mathcal{B}_{2,\infty},$ then for every $\alpha \in [0,1],$
\begin{equation}\label{gamma}
c_\alpha[g]=\int_0^\infty \frac{G(s)}{s^{2-\alpha}}\,ds,
\end{equation}
where $G$ is given by Lemma \ref{T1Dop}.
\end{prop}

\begin{proof}
Observe that $g$ admits the representation \eqref{GtSDop}.
Since
$\Delta_{1+\alpha}^g(z)=z^{1-\alpha}\Delta_2^g(z),$
taking into account the
positivity of $G(s)$ and using  Fubini's theorem, we  infer from  (\ref{GtSDop}) that
\begin{align*}
\int_0^\infty \Delta_{1+\alpha}^g(z)\,dz
=& \int_0^\infty G(s)  \int_0^\infty z^{1-\alpha}\, e^{-zs}\,dz\,ds\\
=& \Gamma(2-\alpha) \int_0^\infty G(s) s^{\alpha-2} \,ds.
\end{align*}
\end{proof}

Similarly, in view of the
boundedness of $G$ on $(0,\infty),$ we obtain the following assertion.

\begin{prop}\label{Sgvz}
Let $g\in \mathcal{B}_1$. Then for all $\delta>0$
and $\alpha\in (0,1),$
\begin{equation}\label{gv11}
\frac{1}{\Gamma(1-\alpha)}
\int_0^\infty \frac{\Delta_2^g(z+\delta)}{z^\alpha}dz=
\int_0^\infty e^{-\delta s}\frac{G(s)}{s^{1-\alpha}}\,ds.
\end{equation}
\end{prop}

Now we relate the functionals $c_0$ and $c_1$ to the mapping $g \to s_g$ introduced in Lemma \ref{LB1}.
\begin{lemma}
Let $g\in \mathcal{B}_{2,\infty}.$
Then
\begin{equation}\label{RS}
c_0[g]+c_1[g]=\int_0^\infty s_g(z)
\,dz,
\end{equation}
and for every $\alpha \in (0,1),$
\begin{equation}\label{rem1}
c_\alpha[g]\le (1-\alpha) c_0[g]+\alpha c_1[g].
\end{equation}
\end{lemma}

\begin{proof}
Integrating by parts and using \eqref{gamma}, we obtain
\[
c_0[g]
=\int_0^\infty \frac{g'(z)+e^{-z}}{z}\,dz
=\int_0^\infty \frac{g'(z)+g(z)}{z}\,dz-c_1[g],
\]
i.e. \eqref{RS} holds.

 From \eqref{gamma}, positivity of $G$
and H\"older's inequality it follows that
\begin{equation*}\label{rem1Vyp}
c_\alpha[g]\le (c_0[g])^{1-\alpha}\cdot (c_1[g])^\alpha.
\end{equation*}
Then a  standard convexity argument
implies \eqref{rem1}.
\end{proof}

We finish this section with the study of  integrability properties for  completely
monotone functions. They will be required in the study of  approximations
for  sectorially bounded holomorphic $C_0$-semigroups.
In particular, the next estimate for powers of completely monotone functions will be useful.
In the following we abbreviate $L^k:=L^k(0,\infty), k \in \N.$
\begin{prop}\label{Pr2}
Let $g\in\mathcal{CM}$  be such that $g \in L^k$ for some $k \in \N.$
Then
\begin{equation}\label{intdg0d}
g^k(z)\le \|g\|_{L^k}\,z^{-1},\qquad z>0,
\end{equation}
and
\begin{equation}\label{intdg0}
g^{k+1}(z)\le k\|g\|_{L^k}\,|g'(z)|,\qquad z>0.
\end{equation}
\end{prop}

\begin{proof}
Let $g$ be given by  (\ref{CMLaplace}) and let $k=1.$
Then by Fubini's theorem,
\begin{equation}\label{L100}
\int_0^\infty s^{-1}\,\nu(ds)=
\int_0^\infty \int_0^\infty e^{-zs}\,\nu(ds)\,dz
=\int_0^\infty g(z)\,dz=\|g\|_{L^1}.
\end{equation}
Since $e^{-zs}\le 1/zs$ for positive $z$ and $s$,
\[
g(z)\le z^{-1}\int_0^\infty s^{-1}\,\nu(ds)
=\|g\|_{L^1}\,z^{-1},\qquad z>0,
\]
so \eqref{intdg0d} holds for $k=1$.
Next, using \eqref{L100}, we have
\[
[g(z)]^2
\le \left(\int_0^\infty s^{-1}e^{-zs}\,\nu(ds)\right)\cdot
\left(\int_0^\infty s e^{-zs}\,\nu(ds)\right)
\le \|g\|_{L^1}\,|g'(z)|,
\]
for every $z>0,$
and
(\ref{intdg0}) is true  for $k=1$ as well.

Assume now that $g \in L^k$  for some $k\in \mathbb N.$
As $g^k\in \mathcal{CM}$ and
\eqref{intdg0d} is valid for $k=1,$ it also holds for all $k \in \N.$
Moreover, using
\eqref{intdg0} with  $k=1$, we obtain
\[
g^{2k}(z)\le\|g^k\|_{L^1}\, |[g^k(z)]'|
= k\|g^k\|_{L^1}\, |g^{k-1}(z)|g'(z)|,\qquad z>0,
\]
and (\ref{intdg0}) holds also for each $k \in \N.$
\end{proof}

Note that if $g \in L^p(0,\infty)$ for some $p \ge 1$ then since $g\le 1$ we have $g \in L^q(0,\infty)$, for all $q \ge p.$
Thus the choice of integers in Proposition \ref{Pr2} does not restrict generality, and it is essentially a matter of convenience.

We will employ the preceding lemma to estimate  moments of sufficiently large powers of functions from $\mathcal {BM}.$
Recall that for $\alpha >0$ and $\beta>0,$
\begin{equation}\label{Class}
\int_0^1 t^{\alpha-1}(1-t)^{\beta-1}dt=
\frac{\Gamma(\alpha)\Gamma(\beta)}{\Gamma(\alpha+\beta)},
\end{equation}
and
\begin{equation}\label{gamma1}
\frac{\Gamma(z)}{\Gamma(z+\beta)}=
z^{-\beta}
\left[1+\frac{\beta(1-\beta)}{2z}+\frac{r_\beta(z)}{z^2}\right],
\qquad z\ge 1,
\end{equation}
where the remainder $r_\beta$ satisfies
\[
r_\beta(z)\le C_{\beta_0},\qquad z\ge 1,\quad \beta\in [0,\beta_0],
\]
see e.g. \cite[Chapter 1.1]{AnAsRo99}.
\begin{lemma}\label{nnk}
Suppose that  $g\in\mathcal{BM} \cap L^k$  for some
$k\in \N$, and  $g(0)=1$. Then there exists
$C(g)>0$ such that for all $\beta\in [0,1]$ and  $n\ge 2k+1, n \in \N,$
\begin{equation}\label{vbeta1}
\int_0^\infty g^n(z) z^\beta\,dz \le C(g) n^{-1-\beta}.
\end{equation}
\end{lemma}

\begin{proof}
First, fix $\beta=0$ and $n\ge k+1.$ Using
\eqref{intdg0} and integrating by parts, we obtain
\begin{equation}\label{exp1120}
\int_0^\infty
g^{n-k-1}(z)g^{k+1}(z)\,dz
\le k\|g^k\|_{L^1}
\int_0^\infty
g^{n-k-1}(z)|g'(z)|\,dz
=\frac{k\|g^k\|_{L^1}}{n-k}.
\end{equation}

Let now $\beta=1$
and $n\ge 2k+1$ be fixed. Then similarly to the above,
\begin{align}
\label{exp112}
\int_0^\infty g^n(z) z\,dz
\le& \frac{k\|g^k\|_{L^1}}{n-k}
\left|\int_0^\infty z
\,d g^{n-k}(z)\right|
=\frac{k\|g^k\|_{L^1}}{n-k}
\int_0^\infty
g^{n-k}(z)\,dz\\
\le& \frac{ k^2\|g^k\|^2_{L^1}}{n-k}
\int_0^\infty
g^{n-2k-1}(z)\,|g'(z)|\,dz
=\frac{ k^2\|g^k\|^2_{L^1}}{(n-k)(n-2k)}.\notag
\end{align}

Note that $n-k\ge \frac{1}{2n}$  and $n-2k\ge \frac{1}{3k}.$ Then, by
\eqref{exp1120}, \eqref{exp112} and H\"older's inequality,
\begin{align*}
\int_0^\infty g^n(z) z^\beta\,dz \le
\left(\int_0^\infty g^n(z)\,dz\right)^{1-\beta}
\left( \int_0^\infty g^n(z) z\,dz\right)^\beta \le C(g) n^{-1-\beta},
\end{align*}
where
\begin{align*}
C(g)=2(3k)^\beta(k\|g^k\|_{L^1})^{1+\beta}\le
6 k^3(1+\|g^k\|_{L^1})^{2}.
\end{align*}
This finishes the proof.
\end{proof}

Lemma \ref{nnk} yields the asymptotics of $(c_\alpha[g_n])_{n \ge 1}$
corresponding to $g$ with good enough
integrability properties.

\begin{thm}\label{Asymp1}
Let $g\in \mathcal{B}_2\cap L^k$
 for some
$k\in \N$.
Then there exists $C(g)>0$ such that for all $\alpha\in [0,1]$ and  $n\in \N,$
\begin{equation}\label{ContAs}
c_\alpha[g_n]=\frac{g''(0)-1}{2n}+r_\alpha(n), \quad \text{and}
\quad
|r_\alpha(n)|\le \frac{C(g)}{n^2}.
\end{equation}
\end{thm}
The proof of the above theorem is rather technical and is postponed to Appendix $2$.

\begin{remark}
Note that in an important case  $g(z)=(1+z)^{-1}$ corresponding to Euler's approximation,
 we have $ g \in \mathcal{B}_2 \cap L^2$ so that $g$ satisfies the assumptions of Theorem \ref{Asymp1}.
 Moreover, in this case  we are able to improve the bounds in \eqref{ContAs}. Indeed,
integrating by parts and using \cite[Formula 2.2.4.24]{Tabl}, we have for every $n \in \N:$
\begin{align*}
\alpha c_\alpha[g_n]=&
\alpha \int_0^\infty \left(\frac{1}{(1+z/n)^n}-e^{-z}\right)\frac{dz}{z^{1+\alpha}}\\
=&\Gamma(1-\alpha)-n^{1-\alpha}\int_0^\infty\frac{dz}{(1+z)^{n+1}z^\alpha}\\
=&\Gamma(1-\alpha)\left(1-\frac{\Gamma(n+\alpha)}{n^\alpha \Gamma(n)}\right).
\end{align*}
Hence  using $\Gamma(2-\alpha)=(1-\alpha)\Gamma(1-\alpha)$, it follows that
\begin{equation}\label{Lim1}
c_\alpha[g_n]
=\frac{1}{\alpha(1-\alpha)}
\left[1-\frac{\Gamma(n+\alpha)}{n^\alpha\Gamma(n)}\right].
\end{equation}
So, passing to the limit in (\ref{Lim1}), we infer that
\begin{align*}
c_0[g_n]=&\frac{1}{\Gamma(n)}\lim_{\alpha\to 0+}\,
\frac{n^\alpha \Gamma(n)-\Gamma(n+\alpha)}{\alpha}=\log n-\psi(n),\\
c_1[g_n]=&\frac{1}{\Gamma(n+1)}
\lim_{\alpha\to 1-}\frac{n^\alpha \Gamma(n)-\Gamma(n+\alpha)}{1-\alpha}\\
=&\psi(n+1)-\log n\\
=&\psi(n)-\log n+\frac{1}{n},
\end{align*}
where  $\psi(z):=\Gamma'(z)/\Gamma(z)$.

Using \cite[Ch. 14, $\S\,5$]{Ficht}, we conclude that
\[
\log n-\psi(n)=\frac{1}{2n}+\frac{\theta_n}{12n^2},
\quad \theta_n\in (0,1),
\]
and then
\begin{equation}\label{psi1}
c_0[g_n]=\frac{1}{2n}+\frac{\theta_n}{12n^2},\quad
c_1[g_n]=\frac{1}{2n}-\frac{\theta_n}{12n^2}
\end{equation}
for every $n\in \N.$ Moreover, by \eqref{rem1} and
\eqref{psi1},
\begin{equation}\label{calpha}
c_\alpha[g_n] \le
\frac{1}{2n}+(1-2\alpha)\frac{\theta_n}{12n^2},\quad \alpha\in [0,1].
\end{equation}
This is in agreement with
a  general estimate of $c_\alpha[g_n]$ in \eqref{ContAs}.
\end{remark}

\section{Approximation of bounded $C_0$-semigroups}\label{approxim}
\subsection{First order approximations}

Now we are able to formulate our first results on approximation of bounded $C_0$-semigroups $(e^{-tA})_{t \ge 0}$ on a Banach space $X$
via completely monotone functions of $A$. We start with an operator norm estimate for $\Delta_0^g(A)$ with $g \in\mathcal B_1.$
The estimate is obtained by a direct application of function-theoretical bounds proved in the previous section and
  elementary properties of the HP-calculus. The approximation formulas with rates will follow then by a simple scaling procedure.

\begin{thm}\label{Psimple1}
Let $-A$ be the generator of a bounded $C_0$-semigroup $(e^{-tA})_{t\ge 0}$ on $X$, and let
 $g \in \mathcal{B}_1$, $g\sim \nu$.
Let  $M:=\sup\limits_{t\geq 0} \|e^{-tA}\|$.
Then
\begin{equation}\label{RowA01x}
\| \Delta_0^g(A)x\| \leq 2M
L[g]  \|A x\|,\qquad
x\in \dom(A),
\end{equation}
and for every $\alpha\in (0,1),$
\begin{equation}\label{RowAx}
\| \Delta_0^g(A)x\| \leq 8M (L[g])^\alpha
\|A^\alpha x\|,\qquad
x\in \dom(A^\alpha).
\end{equation}
\end{thm}

\begin{proof}
Let $\alpha \in (0,1)$ be fixed. Clearly, for every $\delta >0$ the operator $- (A +\delta)$ generates  a bounded
$C_0$-semigroup  $(e^{-t(A+\delta)})_{t \ge 0}$ on $X$ such that $\sup\limits_{t\geq 0} \|e^{-t(A+\delta)}\| \le M.$
We use the HP-calculus to estimate the norm of
$$
(A+\delta)^{-\alpha} \Delta_0^g(A+\delta)
$$
for each  $\delta >0$ and then pass to the limit when $\delta \to 0+.$
Using the product rule for the HP-calculus and the estimate \eqref{Int11} for functions in $\Wip(\mathbb C_+)$,
we infer that
$$
\|[(\cdot+\delta)^{-\alpha} \Delta_0^g(\cdot +\delta)](A)\|=\|(A+\delta)^{-\alpha} \Delta_0^g(A+\delta)\| \le 8M (L[g])^\alpha
$$
for every $\delta>0$.
Since $\dom(A^\alpha) = \dom(A+\delta)^\alpha$ (\cite[Proposition 3.1.9]{Haa2006}), the above inequality
 implies that
\begin{equation}\label{szac0}
\| \Delta_0^g(A+ \delta)x\| \le 8M (L[g])^\alpha  \|(A+\delta)^\alpha x \|,
\qquad x\in \dom(A^\alpha), \quad \delta >0.
\end{equation}
Moreover, by \cite[Proposition 3.1.9]{Haa2006},
$$
\lim\limits_{\delta \to 0^+}(A+\delta)^\alpha x = A^\alpha x,
\qquad x \in \dom(A^\alpha).
$$
Since
$$
\Delta_0^g (z+ \delta) - \Delta_0^g (z) = \int\limits_0^\infty e^{-zs}(e^{-\delta s}-1)\,d\nu(s),
$$
by the dominated convergence theorem,
$$
A_+^1(\mathbb{C}_+)-\lim\limits_{\delta \to 0^+}
\Delta_0^g(\cdot + \delta) = \Delta_0^g(\cdot),
$$
and then
$$
\lim\limits_{\delta \to 0^+}
\,\Delta_0^g(A+\delta) x = \Delta_0^g(A) x,
\qquad x\in X.
$$
Now letting  $\delta \to 0+$ in (\ref{szac0}) we get
 \eqref{RowAx}.

Similarly,
using \eqref{LemA2}
instead  of \eqref{Int11},  we obtain a counterpart of \eqref{szac0} for $x \in \dom (A),$
and then arrive at
\eqref{RowA01x}.
\end{proof}

As a direct implication of Theorem \ref{Psimple1} and Corollary \ref{diffC} we obtain
the following statement where the value of functional $L$ at $g \in \mathcal B_1$ is  replaced by
its estimate in terms of $g'$.

\begin{cor}\label{Psimple2}
Let $-A$ be the generator of a bounded $C_0$-semigroup $(e^{-tA})_{t\ge 0}$ on  $X$, and let
 $g \in \mathcal{B}_1$.
Let $M:=\sup\limits_{t\ge 0} \|e^{-tA}\|$.
If $(g_n)_{n \ge 1}$ are given by \eqref{Defg}, then for all  $n \in \mathbb{N},$  $t>0$ and $x\in \dom(A),$
\begin{equation}\label{RowA001y}
\|g^n\left(tA/n\right)x-e^{-tA}x\|
\le 4 eM
\left(1+\frac{1}{|g'(1/n)|}\right)\sqrt{1+g'(1/n)}\,
 t\|A x\|.
\end{equation}
Moreover, if  $\alpha\in (0,1),$ then for all  $n \in \mathbb{N},$  $t>0$ and $x\in \dom(A^\alpha),$
\begin{equation}\label{RowA011y}
\| g^n\left(tA/n\right)x-e^{-tA}x\| \le 16 e M
\left(1+\frac{1}{|g'(1/n)|}\right)(1+g'(1/n))^{\alpha/2}
t^\alpha \|A^\alpha x\|.
\end{equation}
\end{cor}
(To formulate \eqref{RowA011y} we used that $a^\alpha \le a$ for $a \ge 1$ and $\alpha \in (0,1].$)
\begin{remark}\label{fractional}

In particular, the above statement implies that if
 $g \in \mathcal{B}_1$ is such that
 \begin{equation}\label{polyn}
 g''(\tau) \le c_g \tau^{-1+\gamma}, \qquad \tau \in (0,1],
 \end{equation}
 for a fixed $c_g >0$ and $\gamma \in (0,1),$
and
$(g_n)_{n \ge 1}$ are given by \eqref{Defg}, then for all  $n \in \mathbb{N},$  $t>0$ and $x\in \dom(A),$
\begin{equation*}
\| \Delta^{g_{n}}_{0} (tA)x\|
\le 4 (c_g/\gamma)^{1/2} eM
\left(1+\frac{1}{|g'(1/n)|}\right) n^{-\gamma/2} t\|A x\|.
\end{equation*}
Moreover, if  $\alpha\in (0,1),$ then for all  $n \in \mathbb{N},$  $t>0$ and $x\in \dom(A^\alpha),$
\begin{equation*}
\| \Delta^{g_{n}}_{0} (tA)x\| \le 16 (c_g/\gamma)^{\alpha/2} e M
\left(1+\frac{1}{|g'(1/n)|}\right)
n^{-\alpha\gamma /2}
t^\alpha \|A^\alpha x\|.
\end{equation*}

Note that the functions from $\mathcal{B}_1$ satisfying \eqref{polyn}  are characterized
in terms of their representing measures in Corollary \ref{corNew}.
\end{remark}

\begin{remark}\label{Rem4.3}
Recall that the modulus of continuity $\omega(\epsilon,x)$
of $e^{-\cdot A}x$ over the interval $[0,\epsilon], \epsilon >0,$ is defined as
\[
\omega(\epsilon,x):=\sup
\{\|e^{-tA}x-x\|:\, t\in [0,\epsilon]\}.
\]

Write, as usual,
\[
x=x-S_\epsilon x + S_\epsilon x, \qquad S_\epsilon x:=\frac{1}{\epsilon}\int_0^\epsilon e^{-sA}x\,ds,
\]
where
\[
\|x-S_\epsilon \|\le \omega(\epsilon,x),\qquad
\|A S_\epsilon x \|\le \frac{\omega(\epsilon,x)}{\epsilon}.
\]
From
\eqref{RowA001y} it follows that
\begin{align}\label{ddd}
\| \Delta^{g_n}_{0} (tA)x\|
\le&
\| \Delta^{g_n}_{0} (tA)(x-S_\epsilon x)\|
+\| \Delta^{g_n}_{0} (tA)S_\epsilon x\|
\\
\le& 2M\|x - S_\epsilon x\|\notag \\
+&
4 eM
\left(1+\frac{1}{|g'(1/n)|}\right)\sqrt{1+g'(1/n)}
 t\|A S_\epsilon x\|\notag \\
\le& C M\omega(\epsilon,x)
\left(1+t \frac{(1+g'(1/n))^{1/2}}{|g'(1/n)|\epsilon}\right).\notag
\end{align}
Then, setting
\[
\epsilon=\epsilon_n:=\frac{(1+q'(1/n))^{1/2}}{|g'(1/n)|},
\]
in \eqref{ddd},
we obtain
\begin{equation}\label{RowA001yy}
\| \Delta^{g_n}_{0} (tA)x\| \le
2C M(1+t)\omega(\epsilon_n,x),\qquad x\in X, \quad t >0.
\end{equation}
Thus, one can interpret our approximation results in terms of the modulus continuity, but
we, in general, avoid using this language in the paper. However, the notion of modulus of continuity appear to be
convenient in treating the optimality issues, see the end of this section.
\end{remark}

As an illustration of Corollary \ref{Psimple2}
we derive a partial generalization of one of the main results of L.-K. Chung from
\cite{Chung} (however only for bounded $C_0$-semigroups).
For every $t > 0$ consider
\begin{equation}\label{Cu11}
\varphi_t(u):=\sum_{k=0}^\infty a_k(t)u^k,
\end{equation}
where
\begin{equation}\label{Cu12}
a_k(t)\ge 0,\qquad \sum_{k=0}^\infty
a_k(t)=1,\qquad
\sum_{k=1}^\infty ka_k(t)=t,
\end{equation}
and
\[
\varphi_t''(1)=\sum_{k=2}^\infty k(k-1) a_k(t)<\infty.
\]
By \cite[Theorem 5]{Chung},
if
$(e^{-tA})_{t \ge 0}$ is a bounded $C_0$-semigroup on $X$ then
for every $x\in X,$
\begin{equation}\label{TChu}
e^{-tA} x=\lim_{n\to\infty}\,
[\varphi_t(n(n+A)^{-1})]^nx
\end{equation}
uniformly in $t$ from compacts in $[0,\infty).$
The analysis of the proof
of \cite[Theorem 5]{Chung} reveals that to show that
 the limit in \eqref{TChu} is uniform in $t$
the author uses the additional
assumption
\begin{equation}\label{chung}
\sup_{t\in [0,b]}\,\varphi''_t(1)<\infty \qquad \text{for all}\quad  b>0.
\end{equation}
However, we will not require \eqref{chung} in the following argument.
Using
\eqref{Cu11} and \eqref{Cu12}, let us
define the family of completely monotone functions
\[
g_t(z)=\varphi_t\left(\frac{t}{t+z}\right), \qquad t > 0,
\]
and note that $(g_t)_{t > 0} \subset \mathcal{B}_1.$
Moreover, if
\[
\lim_{u\to 1}\,(\varphi'_t(u)-t)= 0,\qquad \mbox{uniformly in}
\,\, t\in [a,b] \quad \text{for all}\,\, b>a>0,
\]
so that
\[
\lim_{s\to 0}\,(1+g_t'(s))= 0,\qquad \mbox{uniformly in}
\,\, t\in [0,b] \quad \text{for all}\,\, b>a>0,
\]
then by
 \eqref{RowA001y}
we infer that
\[
\|g_t^n(tA/n)x - e^{-tA}x\|\to 0,\qquad n\to\infty,
\]
for every  $x\in X$ uniformly
in $t$ from compacts in  $(0,\infty)$, where
\[
g_t(tA/n)=\varphi_t(n (n +A)^{-1}), \qquad t > 0.
\]
Thus, (\ref{TChu}) holds uniformly in $t$ from compacts in $(0,\infty).$
If one assumes that in addition \eqref{chung} holds, then the uniformity of convergence
in \eqref{TChu} can be extended to compact sets from $[0,\infty)$ (with zero included), see Remark \ref{RemChung} below.

The estimates in Corollary \ref{Psimple2} can be further improved
if $g \in \mathcal B_2.$
The following statement is again a direct implication of Theorem \ref{Psimple1} and Corollary \ref{diffC}, and its proof
is therefore omitted.

\begin{thm}\label{1Th}
Let $-A$ be the generator of a bounded $C_0$-semigroup $(e^{-tA})_{t\ge 0}$ on $X$, and let
$g \in \mathcal{B}_2$.
Let  $M:=\sup\limits_{t\ge 0} \|e^{-tA}\|$.
Then for all  $n \in \mathbb{N}$ and $t>0,$
\begin{equation}\label{RowA0}
\| \Delta_0^{g}(A)x\| \le M
\frac{(g''(0)-1 )}{2} \|A^2 x\|,\qquad
x\in \dom(A^2),
\end{equation}
and
\begin{equation}\label{RowA01}
\| \Delta_0^{g}(A)x\| \le M
(g''(0)-1 )^{1/2} \|A x\|,\qquad
x\in \dom(A).
\end{equation}
Moreover, if  $\alpha\in (0,2),$ then for all  $n \in \mathbb{N},$  and $t>0,$
\begin{equation}\label{RowA}
\| \Delta_0^{g}(A)x\| \le 4M (g''(0)-1)^\alpha
 \|A^\alpha x\|,\qquad
x\in \dom(A^\alpha).
\end{equation}
\end{thm}

Let now $(g_t)_{t > 0}$ be a family of functions from  $\mathcal{B}_2.$
The next result allows us to formulate our approximation results for families $(g_t)_{t > 0}$ rather than just for a fixed
function $g.$ For its proof it suffices to recall
the equation \eqref{h00}
and to scale  Theorem \ref{1Th}
by replacing $A$ with $tA/n.$

\begin{cor}\label{1Th0}
Let $-A$ be the generator of a bounded $C_0$-semigroup $(e^{-tA})_{t\ge 0}$ on $X$, and let
 $(g_t)_{t > 0} \subset \mathcal{B}_2$.
Let  $M:=\sup\limits_{t\ge 0} \|e^{-tA}\|$.
Then for all  $n \in \mathbb{N}$ and $t>0,$
\begin{equation}\label{RowA00}
\| g_t^n(tA/n) x-e^{-tA}x\| \le M
\frac{(g_t''(0)-1 )}{2}\frac{t^2}{n} \|A^2 x\|,\qquad
x\in \dom(A^2),
\end{equation}
\begin{equation}\label{RowA001}
\|g_t^n(tA/n) x-e^{-tA}x\| \le M
(g_t''(0)-1 )^{1/2}\frac{t}{\sqrt{n}} \|A x\|,\qquad
x\in \dom(A),
\end{equation}
and
\begin{equation}\label{RowA011}
\| g_t(tA/n)^n x-e^{-tA}x\| \le 4M \left((g_t''(0)-1)
\frac{t^2}{n} \right)^\frac{\alpha}{2} \|A^\alpha x\|,\quad
x\in \dom(A^\alpha).
\end{equation}
\end{cor}
\begin{remark}\label{RemChung}
Under further assumptions on $(g_t)_{t > 0}$ (e.g. $\sup_{t \in (0,b]} g_t''(0) <\infty$ for every $b >0$) one can, of course, replace $g_t''(0)-1$ in the corollary above by a constant  depending only on $b.$
Thus we can obtain the convergence of approximation formulas with rates uniform in $t$ from compacts in $[0,\infty).$
In particular, in this way, one can show that the convergence in \eqref{TChu} is uniform in $t$ from compacts from $[0,\infty)$
if \eqref{chung} is true.
To simplify our formulations
we omit this obvious improvement here and in the following approximation results. However, it is instructive to have in mind that
the uniformity of convergence present in standard approximation results holds true in our considerations too.
\end{remark}
 Corollary \ref{1Th0} comprises a number of approximation formulas and provide them with optimal convergence rates.
 In particular, it covers \cite[Theorem 1.3]{GTJFA}, and thus the classical Dunford-Segal, Yosida and Euler formulas in
\cite[Corollary 1.4]{GTJFA}.
It also includes formulas falling outside of the scope of \cite{GTJFA}, e.g. Kendall's approximation
corresponding to
$$g_t(z)=1-t+t\exp(-z/t),   \qquad t\in (0,1]. $$
 and more generally any $g_t$ with representing measures having compact supports.
 (Recall that the assumption $t>0$ in Corollary \ref{1Th0} and in the subsequent results
 is a matter of convenience, and it can be replaced by e.g. $t \in (0,a], a>0.$)
 An illustration of  Corollary \ref{1Th0} of the same nature is provided by a spline approximating sequence, and is discussed below.
\begin{example}\label{spline}
Let
\[
g(z):=\int_0^1 e^{-2zs}\,ds=\frac{1-e^{-2z}}{2z},\qquad z>0.
\]
Note that $g\in \mathcal{B}_2$, $g''(0)=4/3,$ and that $g$ is not an exponential of a Bernstein function
(since the representing measure of $g$ has a compact support, see a discussion in the introduction).
The functions $g^n(\cdot/n), n \in \N,$ can be represented in terms of the classical $B$-splines $B_n: \mathbb{R} \to \mathbb{R}$, $n \in \N \cup\{0\},$
defined recursively:
\begin{eqnarray*}
B_0(x)=
\begin{cases}
1,& \quad  x\in [0,1) , \\
0,& \quad  x\not \in [0,1),
\end{cases}
\end{eqnarray*}
and
$$
B_n(x) = \frac{x}{n}\, B_{n-1}(x)+ \frac{n+1-x}{n}\, B_{n-1}(x-1),
\qquad n\in \N.
$$
See \cite[Chapter 9]{Boor} for more details.
Note that since
\begin{equation}\label{zeroB}
B_{n-1}(x)=0\quad \mbox{if}\quad  x\in  (-\infty,0)
\cup (n,\infty),
\end{equation}
and
\[
 B_n'(x) = B_{n-1}(x) - B_{n-1}(x-1),
\]
one has
\begin{equation}\label{recInt}
B_n(x) = \int_0^1 B_{n-1}(x-s)\, ds,
\qquad n\in \N.
\end{equation}

Furthermore,
\begin{equation}\label{SpL}
g^n(z) = \int\limits_0^\infty e^{-2zs} B_{n-1}(s)\, ds,
\quad n\in \N.
\end{equation}
Indeed, arguing by induction, remark that
\[
g(z) = \int_0^1 e^{-2zs}\,ds =\int_0^1 e^{-2zs}B_0(s)\,ds =\int_0^\infty e^{-2zs}B_0(s)\,ds.
\]
If  \eqref{SpL} holds for  $n\in \N,$ then
using
\eqref{zeroB} and \eqref{recInt}, we obtain
\begin{align*}
g^{n+1}(z)
=&\int_0^1 \int_{-s}^\infty e^{-2z(s+t)} B_{n-1}(t)\,dt\,ds
= \int_0^\infty e^{-2z\tau} \int_0^1 B_{n-1}(\tau-s)\,ds\,d\tau\\
=& \int_0^\infty e^{-2z\tau} B_n(\tau)\,d\tau,
\end{align*}
hence  \eqref{SpL} is true for $n+1$, i.e. \eqref{recInt} holds.
So, if $g_n(z)=g^n(z/n), n \in \N,$
ans $-A$ is the generator of a bounded $C_0$-semigroup, then by the HP-calculus,
\[
g_n(tA/n)
= n \int_0^1 e^{- 2 s tA} B_{n-1}(ns)\,ds,\qquad n\in \N.
\]

Now Corollary \ref{1Th0} implies that
for all  $n \in \mathbb{N},$  $t>0,$ and $\alpha \in (0,2],$
\begin{align*}
\left\|n  \int_0^1 B_{n-1}(ns)e^{-2 s t A}x \,ds
-e^{-tA}x
\right\| \le M\frac{t}{\sqrt{3n}} \|A x\|,& \qquad
x\in \dom(A),
\\
\left\| n \int_0^1 B_{n-1}(ns)e^{-2 s t A}x \,ds
-e^{-tA}x
\right\|
\le 4M \left(\frac{t^2}{3n}\right)^\frac{\alpha}{2} \|A^\alpha x\|,& \qquad x\in \dom(A^\alpha),\\
\left\|n  \int_0^1 B_{n-1}(ns)e^{-2 s t A}x \,ds
-e^{-tA}x
\right\| \le M
\frac{t^2}{6 n} \|A^2 x\|,& \qquad
x\in \dom(A^2).
\end{align*}
\end{example}

\subsection{Higher order approximations}
Our functional calculus approach proves to be efficient for deriving also
the second order approximations formulas with rates.
Given the first order formulas, the derivation becomes comparatively straightforward.

Let us define
\begin{align}
d_0[g]:=&
\int_0^\infty
(1-s)^2 G(s)\,ds, \quad
g\in \mathcal{B}_4,
 \label{defd01}\\
d_1[g]:=&\int_0^\infty
\frac{(1-s)^2 (1+s)}{s^2}\,G(s)\,ds,\quad
g\in \mathcal{B}_{4,\infty},
\label{d1F}
\end{align}
where $G$ is given   by \eqref{dopp}.
By \eqref{GtSDop} we have
\begin{equation}\label{defd3}
d_0[g]=
\lim_{z\to 0+} \left(\frac{d}{dz}+1\right)^2 \Delta^g_2(z)=\frac{-3+6g''(0)+4g'''(0)+g''''(0)}{12}.
\end{equation}
Similarly, using again \eqref{GtSDop},
\begin{align*}
d_1[g]=&\int_0^\infty (s - 1
-s^{-1}+s^{-2})\,G(s)\,ds\\
=&-b[g]-c_1[g]+c_0[g],
\end{align*}
where $b[g]$ and $c_j[g], j=0,1,$ are defined by \eqref{defab2} and \eqref{cgg}, respectively.

We are now able to formulate our higher order approximation results for bounded $C_0$-semigroups on Banach spaces.
To our knowledge, apart from \cite{Pf85} and {\cite{Viel1}, higher order approximation formulas for $C_0$-semigroups and their corresponding convergence rates have not been addressed in the literature. As in the previous section, given a bounded $C_0$-semigroup $(e^{-tA})_{t\ge 0}$ on $X,$
we start with a norm
estimate for $\Delta_0^g(A)-2^{-1}(g''(0)-1)e^{-A}A^2, g \in \mathcal B_4,$ on appropriate domains. The approximation formulas will then be derived
by scaling as for the first order approximations considered above.

\begin{thm}\label{1Th+}
Let $-A$ be the generator of a bounded $C_0$-semigroup $(e^{-tA})_{t\ge 0}$ on $X$, and let
 $g \in \mathcal{B}_4$.
Denote $M:=\sup\limits_{t\ge 0} \|e^{-tA}\|$, and suppose that  $b$ and $d_0$ are given by
\eqref{defab2} and \eqref{defd3}, respectively.
Then for every $x\in \dom(A^3),$
\begin{align}\label{RowA0+}
\| g(A)x-e^{-tA}x&-2^{-1}(g''(0)-1)e^{-A}A^2x\|\\
 &\le M\left(\frac{(g''(0)-1)d_0[g]}{2}
\right)^{1/2}\|A^3x\|.\notag
\end{align}
Moreover,
for every $x\in \dom(A^4),$
\begin{align}\label{RowA01+}
\| g(A)x-e^{-tA}x&-2^{-1}{(g''(0)-1)}e^{-A}A^2x\|
\\
& \le M \left(|b[g]|\|A^3x\|
+d_0[g]\|A^4x\|\right).\notag
\end{align}
\end{thm}

\begin{proof}
Note that by  Lemma \ref{T1Dop}, for every $x\in \dom(A^3),$
\begin{align*}
g(A)x-e^{-A}x=&
\int_0^\infty G(s)e^{-sA}A^2x\,ds \\
=&a[g]e^{-A}A^2x
+\int_0^\infty [e^{-sA}-e^{-A}]A^2x  G(s)\,ds,
\end{align*}
where $a[\cdot]$  is defined by \eqref{defab1}.
Since
\[
e^{-sA}x-e^{-A}x=-\int_1^s e^{-tA}Ax\,dt,\qquad x\in \dom(A^3),
\]
we infer that
\begin{equation}\label{340}
g(A)x-e^{-A}x-a[g]e^{-A}A^2x
=-\int_0^\infty
\left(\int_1^s
e^{-t A}A^3x\,dt\,\right)\,
G(s)\,ds.
\end{equation}
Hence,
\begin{align*}
\| \Delta_0^g(A)x-a[g]e^{-A}A^2x\|
\le&
M\|A^3x\|\int_0^\infty
|1-s|
G(s)\,ds
\\
\le& M\|A^3x\|\left(\int_0^\infty
(1-s)^2
G(s)\,ds
\int_0^\infty
G(s)\,ds\right)^{1/2}\\
=& M\|A^3x\|\left(a[g]
d_0[g]\right)^{1/2}.
\end{align*}
It remains to recall that  $a[g]=\frac{g''(0)-1}{2}$ by \eqref{defab3}.

If moreover $x\in \dom(A^4),$ then integrating by parts, we obtain
\begin{align*}
g(A)x-e^{-A}x=&
\int_0^\infty e^{-sA}A^2x G(s)\,ds\\
=&a[g]e^{-A}A^2x-e^{-A}A^3x
\int_0^\infty (s-1)
 G(s)\, ds\\
+&\int_0^\infty [e^{-sA}-e^{-A}+(s-1)e^{-A}A] A^2x G(s)\,ds.
\end{align*}
Taking into account that
\[
e^{-sA}x-e^{-A}x+(s-1)e^{-A}Ax
=\int_1^s (s-t)e^{-tA}A^2x\,dt,
\]
we have
\begin{align}
g(A)x-e^{-A}x-a[g]e^{-A}A^2x
=&-b[g]e^{-A}A^3x \label{34}\\
+&\int_0^\infty
\left(\int_1^s
(s-t)e^{-t A}A^4 x\,dt\,\right)\,
G(s)\,ds. \notag
\end{align}
Moreover,
\begin{align}\label{2180}
\left\|\int_0^\infty
\left(\int_1^s
(s-t)e^{-t A}A^4x\,dt\,\right)\,ds\right\|
\le&
M\|A^4x\|\int_0^\infty
\left(\int_1^s
(s-t)\,dt\,\right)\,
G(s)\,ds  \\
=&M\frac{\|A^4x\|}{2}\int_0^\infty
(s-1)^2\,G(s)\,ds \notag\\
\le&
M d_0[g]\|A^4x\|, \notag
\end{align}
and the estimate
(\ref{RowA01+}) follows.
\end{proof}

Replacing $b[g_n]$ and $d_0[g_n]$ by their expressions in terms of the derivatives of $g$ at zero,
 we formulate below Theorem \ref{1Th+} in a more explicit form.

\begin{cor}\label{1Th+C}
Let $-A$ be the generator of a bounded $C_0$-semigroup $(e^{-tA})_{t\ge 0}$ on $X$, and let
$(g_t)_{t > 0} \subset \mathcal{B}_4$.
Let  $M:=\sup\limits_{t\ge 0} \|e^{-tA}\|.$
Then for all  $t>0$, $n\in \N,$ and $x\in \dom(A^3),$
\begin{align}\label{RowA0++}
\|
g_t^n(tA/n) x-e^{-tA}x&-(2n)^{-1}(g_t''(0)-1) t^2e^{-tA}A^2x\|\\
\le& M C(g_t) t^3 n^{-3/2}\|A^3x\|,\notag
\end{align}
where
\[
C(g_t)=\left(\frac{(g_t''(0)-1)(g_t''''(0)-1)}{2}\right)^{1/2}.
\]
Moreover,
for all  $t>0$, $n\in \N,$ and  $x\in \dom(A^4),$
\begin{align}\label{RowA01++}
\|  g_t^n(tA/n) x-e^{-tA}x&-(2n)^{-1}{(g_t''(0)-1)} t^2e^{-tA}A^2x\|\\
\le&
M C_1(g_t)t^3n^{-2}(\|A^3x\|
+t\|A^4x\|),\notag
\end{align}
where
\[
C_1(g_t)=g_t''''(0)-1.
\]
\end{cor}

\subsection{Optimality  for (general) $C_0$-semigroups}

We finish this section  with a  discussion of optimality for the obtained approximation rates.
A less general version of the following statement was proved  in \cite[Corollary 7.5]{GTJFA}
by means of the spectral mapping theorem for the HP-calculus (Theorem \ref{spmapping})
and certain estimates for Bernstein functions.
Here we propose a slightly different argument based on Corollary \ref{1Th+C}.

\begin{prop}\label{Opt1}
Let $-A$ be the generator of a bounded $C_0$-semigroup $(e^{-tA})_{t\ge 0}$ on a Banach space $X$ such that $\overline{\ran}(A)=X,$
and let $g\in \mathcal{B}_2$, $g(z)\not\equiv e^{-z}$.
If $\{|s|:\,s\in \R, is\subset \sigma(A)\}=\R_{+}$, then there exists $c>0$  such that for all
$\alpha\in (0,2]$, $t>0$, and $n\in \N,$
\[
\|A^{-\alpha}(g^n(tA/n)-e^{-tA})\|\ge c \left(\frac{t^2}{n}\right)^{\alpha/2}.
\]
\end{prop}

\begin{proof}
Let $\alpha\in (0,2]$ be fixed.
Using Corollary \ref{1Th+C} for scalar functions,
let
$t>0$, $n\in \N$ and $s\not=0$ be such that if
\[
\epsilon =\frac{t|s|}{\sqrt{n}},
\]
then $\epsilon$ is so small that $\epsilon \in (0,1)$ and
\[
|g^n(ist/n)-e^{-ist}|\ge c\left(\frac{t s}{\sqrt{n}}\right)^2=
c\epsilon^2.
\]
Then, by (the spectral mapping)  Theorem \ref{spmapping},
we have
\begin{align*}
\|A^{-\alpha}(g^n(tA/n)-e^{-tA})\|\ge&
\sup_{s\in \R}\,||s|^{-\alpha}(g^n(ist/n)-e^{-ist})|\\
\ge& c  |s|^{-\alpha}\epsilon^2=c\epsilon^{2-\alpha}\left(\frac{t^2}{n}\right)^{\alpha/2},
\end{align*}
for some constant $c >0.$
\end{proof}
It is easy to construct concrete examples of generators satisfying the assumptions of Proposition \ref{Opt1},
e.g. one may consider appropriate multiplication operators on $L^2(\mathbb R).$

Let us now address the  optimality
of the
estimate (\ref{RowA0++})
from
Corollary  \ref{1Th+C}. In view of substantial technical difficulties
the optimality will be shown only in a particular setting of $X=C_0(\R),$ a shift semigroup $(e^{-tA})_{t \ge 0}$ on $X,$ and Euler's approximation formula.
\begin{example}\label{Dit}
Let us start from abstract considerations under the assumptions of Corollary \ref{1Th+C}. For $x\in \dom(A^3),$ let
\[
R_{t,n}(A)x:=g^n(tA/n) x-e^{-tA}x-(2n)^{-1}(g''(0)-1) t^2 e^{-tA}A^2x.
\]
Then  (\ref{RowA0++}) asserts  that
\begin{equation}\label{Opt22}
\| R_{t,n}(A)x\|
\le M C(g) t^3 n^{-3/2}\|A^3x\|
\end{equation}
for all  $t>0$ and $n\in \N.$
If  $\omega(\epsilon,x)$ is the modulus continuity of $e^{-tA}x$
and $(S_\epsilon)_{\epsilon >0}$ is the family of averaging operators defined in  Remark \ref{Rem4.3},
then for every $x\in \dom(A^2),$
\[
\|A^2(x-S_\epsilon x)\|\le \omega(\epsilon,A^2x),\qquad
\|A^3 S_\epsilon x \|\le \epsilon^{-1}\omega(\epsilon,A^2x).
\]
Hence, from
(\ref{RowA00}) and (\ref{Opt22}),
it follows that if $x\in \dom(A^2)$ then
\begin{align*}
\|R_{t,n}(A)x\|\le&\|R_{t,n}(A)(x-S_\epsilon x)\|+
\|R_{t,n}(A)S_\epsilon x\|\\
\le& Cn^{-1}t^2[\|A^2x-A^2S_\epsilon x\|+
n^{-1/2}t\|A^3S_\epsilon x\|]\\
\le&  Cn^{-1}t^2[\omega(\epsilon,A^2x)+
n^{-1/2}\epsilon^{-1}t\omega(\epsilon,A^2x)],
\end{align*}
and setting $\epsilon=t/\sqrt{n}$ we have
\begin{equation}\label{Opt3}
\|R_{t,n}(A)x\|\le Ct^2 n^{-1}
\omega(t/\sqrt{n},A^2x),\qquad x\in \dom(A^2).
\end{equation}

Now, following \cite{D71},
 let $X$ be a Banach space
of continuous functions on $\R$, vanishing at both infinities,
denoted by $C_0(\R).$
Define a  $C_0$-semigroup $(e^{-tA})_{t \ge 0}$ on  $C_0(\R)$ by
\[
(e^{-tA}f)(s):=f(s+t),\qquad s\in \R,
\]
and note that $(Af)(s)=-f'(s)$ on a natural domain.

Consider Euler's approximation of $(e^{-tA})_{t \ge 0},$ i.e. choose
$g(z)=1/(1+z)$ in Corollary \ref{1Th+C}.
Then
\[
g^n(z/n)=\int_0^\infty e^{-zs} v_n(s)\,ds,\qquad \text{where} \quad
v_n(s):=\frac{n^n}{(n-1)!}e^{-ns} s^{n-1}, \quad n \in \N.
\]
Making use of  Lemma \ref{T1Dop}
and the proof of Theorem \ref{1Th+},
we infer that for $t>0$ and $f\in \dom(A^2),$
\begin{align*}
(R_{t,n}f)(s)=
(t^2Z_{t, n}A^2 f)(s), \qquad s \in \R,
\end{align*}
where
\begin{align*}
(Z_{t, n}h)(s):=&
\int_0^\infty W_n(\tau)\,(e^{-\tau t A}-e^{-tA})h(s)\,d\tau,\\
W_n(\tau):=&\int_0^\infty |\tau-y|v_n(y)\,dy,\quad \tau>0.
\end{align*}

Let $u\in C_0(\R)$ be given by
\[
u(s)=\begin{cases} 1-|s|,&\qquad |s|\le 1,\\
0,&\qquad |s|>1,
\end{cases}
\]
and let $f\in \dom(A^2)$ be defined as
\[
f=(A^2-I)^{-1}u,
\]
so that $A^2f=u+f.$
Taking into that
$\omega(\epsilon, u)\le \epsilon$ for every $\epsilon\in (0,1),$ we have
\begin{equation}\label{modEps}
\omega(\epsilon, A^2f )\le \omega(\epsilon, u)+\omega(\epsilon, f)\le 2\epsilon,\qquad
\epsilon\in (0,1).
\end{equation}
Note that
\begin{equation}\label{rrr}
(R_{t,n}f)(s)
=t^2\left((Z_{t,n}u)(s)+ (Z_{t,n}f)(s)\right), \qquad s \in \R,
\end{equation}
and (see the proof of Theorem \ref {1Th+})
\begin{equation}\label{zzz}
\|Z_{t,n}f\|\le c M n^{-2}\left(\|Af\|+\|A^2f\| \right),
\end{equation}
for some constant $c >0.$
Setting $t=1$ and $s=-1$ in \eqref{rrr}, write
\[
(Z_{t,n}u)(-1)
=\int_0^\infty W_n(\tau)\,\left(u(\tau-1)-1\right)\,d\tau=I_{1,n}+I_{2,n},
\]
where
\[
I_{1,n}:=-\int_0^2 W_n(\tau)\,|\tau-1| d\tau,\qquad
I_{2,n}:=-\int_2^\infty W_n(\tau)\, d\tau.
\]
Estimating the second integral from above, we obtain that for every $n \in \N,$
\begin{equation}\label{i2}
|I_{2,n}|\le  \int_2^\infty (\tau-1)^2 W_n(\tau)\, d\tau\le
\int_0^\infty (\tau-1)^2 W_n(\tau)\, d\tau\le 2n^{-2}.
\end{equation}
On the other hand, using
\[
\int_1^y (y-\tau)(\tau-1) d\tau=\frac{(y-1)^3}{6},
\]
and estimating the first integral from below we have:
\begin{align*}
|I_{1,n}|\ge \int_1^2 W_n(\tau)\,|\tau-1| d\tau
=& \int_0^1 \int_\tau^\infty (y-\tau) v_n(y) dy \,(\tau-1) d\tau\\
\ge& \int_1^2 v_n(y)\int_1^y (\tau-y)(\tau-1) d\tau\,dy
\\
=&\frac{n^n e^{-n}}{6(n-1)!}\int_0^1 e^{-ns} (1+s)^{n-1}
s^3 ds.
\end{align*}
Now
 Stirling's formula
and the property 
\[
\lim_{n\to\infty}\,
n^2\int_0^1 e^{-ns}(1+s)^{n-1} s^3\,ds=2,
\]
see \cite[p. 81]{Olver},
imply that
\begin{equation}\label{i1}
\lim\inf_{n\to\infty} n^{3/2}|I_{1,n}|
\ge \frac{1}{3\sqrt{2\pi}}.
\end{equation}

So, taking into account \eqref{rrr}, \eqref{zzz}, \eqref{i2} and \eqref{i1}, we infer that
\begin{equation}\label{above}
\lim\inf_{n\to\infty}\,n^{3/2}\|R_{1,n}f\|
\ge \frac{1}{3\sqrt{2\pi}}.
\end{equation}
By \eqref{modEps}, we have
$\sqrt{n}\omega(1/\sqrt{n},A^2f)\le 2,$
so we can
rewrite \eqref{above} in the form
\begin{equation}\label{above11}
\lim\inf_{n\to\infty}\, \frac{n}{\omega(1/\sqrt{n},A^2f)}\|R_{1,n}f\|
\ge \frac{1}{6\sqrt{2\pi}}.
\end{equation}
The inequality \eqref{above11} shows that \eqref{Opt3}
is sharp. Consequently,  \eqref{Opt22} is sharp too.
\end{example}
We believe that the optimality of (\ref{RowA0++}) can be shown in a more general context. However
we feel that the exposition will then be overloaded by unnecessary technicalities, and postpone
a general argument to another occasion.

\section{Approximation of bounded holomorphic $C_0$-semigroups.}\label{holomorph}

As one may expect, the statements above can be  improved in the framework of sectorially bounded holomorphic semigroups.
Recall that
for
$g=e^{-\varphi}$,
where $\varphi$ is a Bernstein function,
 the statements on  approximation
of  sectorially bounded holomorphic $C_0$-semigroups were
deduced in \cite{GTJFA} from
the results on
approximation of general bounded $C_0$-semigroups.
To this aim, given a sectorially bounded holomorphic $C_0$-semigroup
$(e^{-tA})_{t \ge 0}$ on $X,$
we
considered in \cite{GTJFA} a bounded $C_0$-semigroup $(e^{-t \mathcal A})_{t \ge 0},$
\[
\mathcal{A}
=\left(\begin{array}{cc}
A & A\\
0 & A
\end{array}
\right), \qquad \dom (\mathcal A)=\dom (A) \oplus \dom (A),
\]
on the direct sum  $X\oplus X$, and, after certain manipulations, read off approximation results for
$(e^{-tA})_{t \ge 0}$ on $X$ from the ones for $(e^{-t \mathcal A})_{t \ge 0}$ on $X\oplus X.$
(Note that the idea to use operator matrices as above to the study of holomorphic $C_0$-semigroups goes back to \cite{Crandall}.)

Below we propose a more direct and explicit approach.
It does not depend on the approximation results for bounded, not necessarily holomorphic, $C_0$-semigroups, and it improves the results from the previous section in a particular
framework of holomorphic semigroups. Moreover, it extends the results from \cite{GTJFA} by allowing $g \in \mathcal {BM},$ and sometimes offers even better estimates than those in
\cite{GTJFA} as far as the constants are concerned.

For a sectorially bounded holomorphic $C_0$-semigroup $(e^{-tA})_{t \ge 0}$ on $X$ and $\beta \ge 0$ define
(see \eqref{boundbeta})
\begin{equation}\label{mb}
M_\beta:=\sup_{t > 0} \|t^\beta A^\beta e^{-tA}\|<\infty.
\end{equation}
Thus, the constat $M_0$ has a meaning of $M$ from the previous section dealing with general $C_0$-semigroups.

Note that
using moment's inequality,  for all $t >0$ and $\gamma \in (0,1),$ we have
\begin{align*}
t^{1+\gamma} \|A^{1+\gamma}e^{-tA}x\| \le& C(\gamma) t^{1+\gamma} \|A e^{-tA}x\|^{1-\gamma} \|A^2 e^{-tA}x\|^\gamma\\
=& C(\gamma)  \|t A e^{-tA}x\|^{1-\gamma} \|t ^2 A^2 e^{-tA}x\|^\gamma\\
\le& C(\gamma)M_1^{1-\gamma}M_2^\gamma \|x\|.
\end{align*}
Thus, since $M_1^{1-\gamma}M_2^\gamma \le (1-\gamma)M_1 + \gamma M_2$ and, by
\cite[p. 63]{Mart1},  $C(\gamma) \le 3, \gamma \in [0,1],$ we have
\begin{equation}\label{moment}
M_{1+\gamma} \le C(\gamma) (M_1 +M_2)\le 3(M_1+M_2).
\end{equation}

\subsection{First order approximations}
We start with several simple auxiliary estimates for functions of generators of holomorphic semigroups.
\begin{prop}\label{rfun}
Let $r \in \mathcal{BM},$  and denote
\[
r_j(z):=z^jr(z),\qquad j=1,2.
\]
Let $-A$ be the generator of a sectorially bounded holomorphic semigroup $(e^{-tA})_{t\ge 0}$ on $X$.
Then the following estimates hold:
\begin{itemize}
\item [\emph{(i)}] 
$\|r(A)\|\le M_0r(0),\qquad \|r_1'(A)\|\le (M_0+M_1)r(0),$
\item [\emph{(ii)}]\label{BB2}
$\|r_2'(A)x\|\le (2M_0+M_1)r(0)\|Ax\|,\qquad x\in \dom(A),$
\item [\emph{(iii)}]
$\|r_2''(A)\|\le (2M_0+4M_1+M_2)r(0).$
\end{itemize}
\end{prop}
\begin{proof}
The first assertion from \emph{(i)}
is evident. Next, letting $r\sim \gamma$, note that
\[
r_1'(z)=r(z)-z\int_0^\infty se^{-sz}\,\gamma(ds)
\]
and use \eqref{mb} and the product rule for the (extended) HP-calculus.
 It follows that for all  $\delta>0$ and $x\in X,$
\begin{align*}
r_1'(A+\delta)x=&r(A+\delta)x-\int_0^\infty Ase^{-sA}e^{-\delta s}x\,\gamma(ds)\\
-&
\int_0^\infty \delta se^{-sA}e^{-\delta s}x\,\gamma(ds).
\end{align*}
Hence
\[
\|r_1'(A+\delta)x\|\le M_0r(0)\|x\|+M_1 r(0)\|x\|
+M_0\|x\|\int_0^\infty \delta se^{-\delta s}\,\gamma(ds).
\]
Letting  $\delta \to 0+$  and using the dominated convergence theorem, we obtain
the second assertion in \emph{(i)}. The other statements \emph{(ii)} and \emph{(iii)}
can be proved similarly.
\end{proof}

Observe that
\begin{align}\label{rep}
g(z)-e^{-z}
=&[g(z)+2g'(z)+g''(z)]\\
-&2[g'(z)+g''(z)]+[g''(z)-e^{-z}],\notag
\end{align}
and
\begin{equation}\label{rep11}
g(z)-e^{-z}
=[g(z)+g'(z)]-[g'(z)+e^{-z}].
\end{equation}

Proposition \ref{rfun} and Lemma \ref{LB1} enable us to estimate each term in square brackets
from \eqref{rep} and \eqref{rep11}, and thus to estimate $g(A)-e^{-A}$ and
$g(A)x-e^{-A}x,$ $x\in \dom(A)$.

\begin{cor}\label{corD1}
Let
$g\in \mathcal{B}_2$,
and let $-A$ be the generator of a sectorially bounded holomorphic semigroup $(e^{-tA})_{t\ge 0}$ on $X$.
Then the following estimates hold:
\begin{itemize}
\item [\emph{(i)}]
$\|g'(A)+g''(A)\|\le (M_0+M_1)(g''(0)-1),$\\
\item [\emph{(ii)}] $\|g(A)+2g'(A)+g''(A)\|\le M_0 (g''(0)-1),$\\
\item [\emph{(iii)}] $\|e^{-A}-g''(A)\|\le (M_0+2M_1+M_2/2)(g''(0)-1),$\\
\item [\emph{(iv)}] for all $x\in \dom(A),$
\begin{equation}\label{firstA}
\|(g(A)+g'(A))x\|\le M_0(g''(0)-1)\|Ax\|,
\end{equation}
and
\begin{equation}
\label{BB10x}
\|g'(A)x+e^{-A}x\|\le (M_0+M_1/2)(g''(0)-1)\|Ax\|.
\end{equation}
\end{itemize}

\end{cor}

\begin{proof}
Lemma \ref{LB1} implies that if $g\in \mathcal{B}_2,$
then
$r(z):=s_g(z)\in \mathcal{BM}$ and $
r(0)=g''(0)-1.$
Using Proposition \ref{rfun}, \emph{(i)}, where $r_1(z):=g(z)+g'(z)$, we obtain the assertion
\emph{(i)}.

To prove \emph{(ii)} let $g\sim \nu$.
Note that for
\begin{equation}\label{qqq}
q(z):=g(z)+2g'(s)+g''(z)=\int_0^\infty (s-1)^2 e^{-sz}\,\nu(ds), \qquad z>0,
\end{equation}
we have $q \in \mathcal{BM}$ with $q(0)=g''(0)-1$.
Therefore,
\begin{equation}
\|q(A)\|\le M_0 (g''(0)-1),
\end{equation}
that is, \emph{(ii)} is true.
Next, by Lemma \ref{T1Dop}, if
$
r(z):=\Delta_2^g(z)$ then $ r \in \mathcal{BM}$ and $r(0)=\frac{g''(0)-1}{2}.$
Now setting $r_2(z)=g(z)-e^{-z}$ and applying   Proposition \ref{rfun}, \emph{(ii)} and \emph{(iii)} to $r$,
we obtain (\ref{BB10x}).
Finally, the estimate \eqref{firstA} follows from the product rule for the (extended) HP-calculus and Lemma \ref{LB1}.
\end{proof}

\begin{remark}
If $q$ is defined by \eqref{qqq}, then noting that
\[
g(z)-g''(z)=q(z)-2(g'(z)+g''(z)),
\]
and using  Proposition \ref{corD1}, \emph{(ii)} and \emph{(iv)},
one infers that
\begin{equation*}
\|g(A)-g''(A)\| 
\le (2M_0+M_1)(g''(0)-1).
\end{equation*}
\end{remark}

The following theorem is a direct implication of norm estimates for $\Delta_1^g(A)$ on appropriate domains,
given by the HP-calculus.

\begin{thm}\label{tD3}
Let
$g\in \mathcal{B}_2$,
let $-A$ be the generator of a sectorially bounded holomorphic semigroup $(e^{-tA})_{t\ge 0}$ on $X$, and
let $K=3M_0+3M_1+M_2/2$.
Then
\begin{equation}\label{BB30y}
\|g(A)-e^{-A}\|\le K(g''(0)-1),
\end{equation}
and for every $x \in \dom(A),$
\begin{equation}\label{BB10y}
\|g(A)x-e^{-A}x\|
\le (2M_0+M_1/2)(g''(0)-1)\|Ax\|.
\end{equation}
Moreover, for every $\alpha \in (0,1)$ and every $x\in \dom(A^\alpha),$
\begin{equation}\label{assD}
\|g(A)x-e^{-A}x\|\le 3 M_0 K
(g''(0)-1)\|A^\alpha x\|,
\end{equation}
\end{thm}

\begin{proof}
The inequalities
\eqref{BB30y} and \eqref{BB10y}
follow from Corollary  \ref{corD1} and the representations
\eqref{rep} and \eqref{rep11}.

Next, let $\alpha\in(0,1),$  $\beta=1-\alpha,$ and $\delta>0$.
Combining moment's inequality for the sectorial operator $A+\delta$
with \eqref{BB30y} and \eqref{BB10y} and employing the (extended) HP-calculus, we obtain for any $x\in X:$
\begin{align*}
\|(A+\delta)^{-\alpha}\Delta_0^g(A+\delta)x\|=&
\|(A+\delta)^\beta[(A+\delta)^{-1}\Delta_0^g(A+\delta)]x\|\\
\le& 3M_0\|\Delta_0^g(A+\delta)x\|^\beta
\|(A+\delta)^{-1}\Delta_0^g(A+\delta)x\|^{1-\beta}\\
\le& 3M_0 K^\beta(2M_0+M_1/2)^{1-\beta}
(g''(0)-1)\|x\|\\
\le& 3M_0 K(g''(0)-1)\|x\|.
\end{align*}
Thus, for all $x\in \dom(A^\alpha),$
\[
\|g(A+\delta)x-e^{-(A+\delta)}x\|
\le 3M_0 K
(g''(0)-1)\|(A+\delta)^\alpha x\|.
\]
Letting now  $\delta\to 0+$,
assertion (\ref{assD}) follows.
\end{proof}

Now by scaling we derive our first order approximation formulas with rates
for sectorially bounded holomorphic $C_0$-semigroups on Banach spaces.

\begin{cor}\label{tD3cor}
Let
$(g_t)_{t > 0}\subset  \mathcal{B}_2$,
and let $-A$ be the generator of a sectorially bounded holomorphic
$C_0$-semigroup $(e^{-tA})_{t \ge 0}$ on $X$.
Then for all $n\in \N$ and $t>0,$
\begin{equation}\label{BB30yc}
\|g_t^n(tA/n)-e^{-tA}\|\le K \frac{(g_t''(0)-1)}{n},
\end{equation}
where $K=3M_0+3M_1+M_2/2.$
Moreover for all $n\in \N,$ $t>0,$ and $x\in \dom(A),$
\begin{equation}\label{BB10yc}
\|g_t^n(tA/n)x-e^{-tA}x\|
\le (2M_0+3M_1/2)\frac{(g_t''(0)-1)}{n}t\|Ax\|,
\end{equation}
and for all $n\in \N,$ $t>0,$ $\alpha \in (0,1),$ and $x\in \dom(A^\alpha),$
\begin{equation}\label{momento}
\|g_t^n(tA/n)x-e^{-t A}x\|
\le 3M_0 K
\frac{(g_t''(0)-1)}{n}t^\alpha\|A^\alpha x\|.
\end{equation}
\end{cor}

The above corollary can be essentially sharpened with respect to constants
in the right hand side of \eqref{momento}. The improvement requires
more involved arguments and finer function classes, and is contained in the following statement .

\begin{thm}\label{ThA}
Let $-A$ be the generator of a sectorially bounded holomorphic
$C_0$-semigroup $(e^{-tA})_{t\ge 0}$ on $X,$
and let $g\in \mathcal{B}_{2,\infty}$.
Then for  all $t >0,$ $\alpha\in [0,1]$ and $x\in \dom(A^\alpha),$
\begin{equation}\label{b1}
\|g^n(t A/n)x-e^{-tA}x \|\le
M_{2-\alpha} c_\alpha[g]\|A^\alpha x\|,
\end{equation}
where $c_\alpha$ and $M_{2-\alpha}$ are given by \eqref{cgg} and \eqref{mb}, respectively.
\end{thm}

\begin{proof}
We consider the cases  $\alpha\in [0,1)$  and $\alpha=1$ separately.
Let $\alpha\in [0,1)$ be fixed.
By the product rule for the (extended) HP-calculus,
for every  $\delta>0,$
\begin{equation}\label{l11}
\Delta_0^g(A+\delta)=(A+\delta)^2
[\Delta_2^g(z+\delta)](A),
\end{equation}
where $\Delta_2^g(\cdot+\delta) \in \mathcal {BM},$
\begin{equation}\label{l12}
\Delta_2^g(z+\delta)=
\int_0^\infty e^{-zs}e^{-\delta s}G(s)\,ds,\,\, \, \, \text{and} \,\,\,\,
\|\Delta_2^g(z+\delta)\|_{\Wip(\C_+)}
=\Delta_2^g(\delta).
\end{equation}
Then, by (\ref{l11}) and (\ref{l12}), for every
for $x\in \dom(A^\alpha)$,
\begin{equation}\label{Equa}
\|\Delta_0^g(A+\delta)x\|
\le \int_0^\infty e^{-\delta s} G(s)\|(A+\delta)^{2-\alpha}e^{-sA}(A+\delta)^\alpha x\|\,ds.
\end{equation}
To estimate the right hand side of \eqref{Equa}, recall that $e^{-sA}(X)\subset \dom(A)$, $s>0$, and
by
\cite[Proposition 3.1.7]{Haa2006} for all $\beta\in [0,1]$ and $\delta>0,$
\[
\|(A+\delta)^\beta x-A^\beta x\|\le C_\beta(A)\delta^\beta\|x\|.
\]
Employing (\ref{b1}) we obtain
\begin{align}
 \int_0^\infty e^{-\delta s} &G(s) \|(A+\delta)^{2-\alpha}e^{-sA}  (A+\delta)^\alpha x \|\,ds \label{abov}\\
\le& \int_0^\infty e^{-\delta s} G(s)\|(A+\delta)A^{1-\alpha}e^{-sA}(A+\delta)^\alpha x \|\,ds \notag\\
+& \int_0^\infty e^{-\delta s} G(s)\|[(A+\delta)^{1-\alpha}-A^{1-\alpha}](A+\delta)e^{-sA} (A+\delta)^\alpha x \|\,ds\notag \\
\le& \|(A+\delta)^\alpha x\|\int_0^\infty e^{-\delta s} G(s)\left[\frac{M_{2-\alpha}}{s^{2-\alpha}}+
\delta \frac{M_{1-\alpha}}{s^{1-\alpha}}\right]\,ds \notag \\
+&\|(A+\delta)^\alpha x\|C_{1-\alpha}(A)\delta^{1-\alpha}
\int_0^\infty e^{-\delta s} G(s)[M_1s^{-1}+M_0\delta]\,ds.\notag
\end{align}
Then, using  (\ref{gamma}),
(\ref{gv11}) and (\ref{abov}), we infer from (\ref{Equa})
that
\begin{equation}\label{D}
\|\Delta_0^g(A+\delta)x\|\le (M_{2-\alpha}c_\alpha[g]+N_\alpha(\delta))\|(A+\delta)^\alpha x\|, \quad x\in \dom(A^\alpha),
\end{equation}
where
\begin{align*}
N_\alpha(\delta)=&
\frac{M_{1-\alpha}}{\Gamma(1-\alpha)}
\delta\int_0^\infty \frac{\Delta_2^g(z+\delta)}{z^\alpha}\, dz\\
+&C_{1-\alpha}(A)M_1\delta^{1-\alpha}
\int_0^\infty \Delta_2^g(z+\delta)\,dz
+C_{1-\alpha}(A)
M_0\delta^{1-\alpha}.
\end{align*}
Observe now that
\begin{equation}\label{lll}
\lim_{\delta\to 0+}\,N_\alpha(\delta)=0.
\end{equation}
Indeed,
by (\ref{pos3}) and the inequality $z+1\ge z^\nu, z>0,$ with  $\nu=(1-\alpha)/2\in (0,1]$,
we have
\begin{align*}
\frac{\delta}{2}\int_0^\infty
\frac{\Delta^g_2(z+\delta)}{z^\alpha}\,dz\le&
\delta\int_0^\infty
\frac{dz}{z^\alpha(z+\delta)(z+1)}\\
\le& \delta\int_0^\infty
\frac{dz}{z^{\alpha+\nu}(z+\delta)}\\
=&
\delta^{(1-\alpha)/2}\int_0^\infty
\frac{dz}{z^{(1+\alpha)/2}(z+1)},
\end{align*}
and \eqref{lll} follows.
Thus, letting  $\delta\to 0+$ in (\ref{D}),
 we conclude  that
\[
\|\Delta_0^g(A)x\|
\le
M_{2-\alpha}c_{\alpha}(g)\|A^\alpha x\|,\qquad x\in \dom(A^\alpha).
\]

Finally, if $\alpha=1$ and $x\in \dom(A)$,
then taking into account (\ref{gamma}), we obtain:
\begin{align}
\|\Delta_0^g(A+\delta)x\|\label{Eqq1}
\le& \int_0^\infty e^{-\delta s} G(s)\|(A+\delta)e^{-sA}(A+\delta) x\|\,ds\\
\le& \|(A+\delta)x\| \int_0^\infty e^{-\delta s} G(s)[M_1s^{-1}+\delta M_0]\,ds \notag\\
\le& \|(A+\delta)x\|\left\{M_1c_1[g]+ M_0\delta \int_0^\infty e^{-\delta s} G(s)\,ds\right\}.\notag
\end{align}
Thus  letting  $\delta\to 0+$ in (\ref{Eqq1}) and  using (\ref{rem}),
we get (\ref{b1}) for $\alpha=1$ as well.
\end{proof}

Recall that, according to Corollary  \ref{tD3cor}, \eqref{momento} holds with a constant $3 M_0(3M_0+3M_1+M_2/2),$ while
Theorem \ref{ThA}, in view of  \eqref{moment},  Theorem \ref{Asymp1} and $M_0\ge 1$, provides a much better constant $3(M_1+M_2).$ Note that, curiously, $M_0$ does not enter the estimate
of Theorem \ref{ThA}.

Theorem \ref{ThA}
yields the following sharp approximation formula.

\begin{cor}\label{ThACor}
Let $-A$ be the generator of a sectorially bounded holomorphic
$C_0$-semigroup $(e^{-tA})_{t\ge 0}$ on $X$, and
let $g\in \mathcal{B}_{2}\cap L^k$  for some
$k\in \N$.
Then there exists  $C(g)>0$ such that for all $\alpha\in [0,1]$ and $x\in \dom(A^\alpha),$
\begin{equation}\label{b11}
\|g^n(t A/n)x-e^{-tA}x\|\le
M_{2-\alpha} \frac{g''(0)-1}{2n}
\left(1+C(g)n^{-1}\right)t^\alpha\|A^\alpha x\|.
\end{equation}
\end{cor}

\subsection{Higher order approximations and sharp constants}

In this subsection we obtain a version of the second order approximation formulas from  Theorem \ref{1Th+}
for sectorially  bounded holomorphic semigroups.
While in Theorem \ref{1Th+} the constants in approximation rates were controlled by functionals $d_0$ and $b$ introduced
in Section \ref{functions}, in the setting of holomorphic semigroups
 the functional $d_1$ substitutes
 $d_0.$

\begin{thm}\label{ThAx}
Let $-A$ be the generator of a sectorially bounded holomorphic
$C_0$-semigroup $(e^{-tA})_{t\ge 0}$ on $X$.
Assume that  $g\in \mathcal{B}_4\cap L^k$ for some $k\in\N$.
Then for all $\alpha\in [0,1]$ and $x\in \dom(A^\alpha),$
\begin{equation}\label{b1x}
\left \|g(A)x-e^{-A}x
-\frac{g''(0)-1}{2} A^2e^{-A}x \right\|
\le K(g, \alpha)\|A^\alpha x\|.
\end{equation}
where $$K(g, \alpha)=
|b[g]| M_{3-\alpha}
+2^{-1}d_1[g] M_{4-\alpha},$$
and $d_1$ and $M_\alpha$ are given by \eqref{d1F}
 and \eqref{mb}, respectively.
\end{thm}

\begin{proof}
Fix $\alpha \in [0,1].$
By \eqref{34}
 we have
\begin{align*}
g(A)x-e^{-A}x
-a[g]A^2e^{-A}x
=&b[g]A^3e^{-A}x\\
+&\int_0^\infty
\left(\int_1^s
(s-t)A^4e^{-t A}x\,dt\,\right)\,
G(s)\,ds.
\end{align*}
Then,
\begin{align*}
\|g(A)x-e^{-A}x
-a[g]&A^2e^{-A}x\|
\\
\le& |b[g]| M_{3-\alpha}\|A^\alpha x\|
+M_{4-\alpha}\|A^\alpha x\|\int_0^\infty
I_\alpha(s)\,
G(s)\,ds,
\end{align*}
where
$$
I_\alpha(s):=\int_1^s (s-t)t^{\alpha-4}\,dt = \frac{(2-\alpha)s^3-(3-\alpha)s^2+s^\alpha}{(3-\alpha)(2-\alpha)s^2}.
$$
Since
\[
s^\alpha\le \alpha s+(1-\alpha), \qquad s>0,
\]
we have
\begin{align*}
I_\alpha(s)\le&
\frac{(2-\alpha)s^3-(3-\alpha)s^2+\alpha s+(1-\alpha)}{2(2-\alpha)s^2}\\
=&\frac{(1-s)^2((2-\alpha)s+(1-\alpha)}{2(2-\alpha) s^2}\\
\le&
\frac{(1-s)^2(1+s)}{2s^2}.
\end{align*}

Therefore,
\begin{align*}
\|g(A)x-e^{-A}x
-a[g]&A^2e^{-A}x\|\\
\le&
|b[g]| M_{3-\alpha}\|A^\alpha x\|
+ \frac{d_1[g] M_{4-\alpha}}{2}\|A^\alpha x\|.
\end{align*}
\end{proof}

Now Corollary
\ref{Corgn} makes it possible to replace the constants in Theorem \ref{ThAx} by
their estimates in terms of $g$ and $M_\beta$.

\begin{cor}\label{ThAxC}
Let $-A$ be the generator of a sectorially bounded holomorphic
$C_0$-semigroup $(e^{-tA})_{t\ge 0}$ on $X$.
Assume that $g\in \mathcal{B}_4 \cap L^k$ for some $k\in \N$.
If $\alpha\in [0,1]$ then
there exists  $C(g)>0$ such that
for all $n\in \N$, $t>0$,  and  $x\in \dom(A^\alpha),$
\begin{align}\label{b1xC}
\|g^n(tA/n)x-e^{-tA}x
-& a[g]n^{-1}t^2A^2e^{-tA}x\|\\
\le&
C(g)[M_{3-\alpha}
+M_{4-\alpha}]\frac{t^\alpha}{n^2} \|A^\alpha x\|,
\end{align}
where $a[g]=(g''(0)-1)/2$.\notag
\end{cor}

 Corollary \ref{ThAxC}
generalizes the corresponding result from \cite{Viel1}, where a second order Euler's approximation (with  $\alpha=0$) for holomorphic semigroups has been studied
using a completely different technique.

\subsection{Optimality for holomorphic $C_0$-semigroups}

As for general $C_0$-semigroups,  our approximation
formulas for holomorphic semigroups are sharp, as the next proposition shows.

\begin{prop}\label{Opt2}
Let $-A$ be the generator of a sectorially bounded holomorphic $C_0$-semigroup $(e^{-tA})_{t\ge 0}$ on a Banach space $X$ such that $\overline{\ran}(A)=X$ and let $g\in \mathcal{B}_4$, $g(z)\not\equiv e^{-z}$.
If $\R_{+}\subset \sigma(A)$, then there exists $c>0$  such that for all
$\alpha\in [0,1],$  $t>0$, and $n\in \N,$
\[
\|A^{-\alpha}(g^n(tA/n)-e^{-tA})\|\ge c \frac{t^\alpha}{n}.
\]
Moreover, if in addition $b[g]\neq 0$,  then
there is  $c>0$  such that for all
$\alpha\in [0,1],$  $t>0$, and $n\in \N,$
\[
\|A^{-\alpha}(g^n(tA/n)-e^{-tA}
-a[g]n^{-1}t^2A^2e^{-tA})\|
\ge
c\frac{t^\alpha}{n^2}.
\]
\end{prop}

\begin{proof}
Let $\alpha\in [0,1]$ be fixed.
Employing Corollary \ref{1Th+C} (for scalar functions) with  $s=1/t$
we conclude that  there exists $c>0$ such that
\[
|g^n(1/n)-e^{-1}|\ge cn^{-1}
\]
for all $n \in \N.$
Then, by Theorem \ref{spmapping},
we have
\[
\|A^{-\alpha}(g^n(tA/n)-e^{-tA})\|\ge
\sup_{\lambda \in \R_{+}}\,|\lambda^{-\alpha}(g^n(\lambda t/n)-e^{-\lambda t})|
\ge c  t^{\alpha}n^{-1},
\]
for all $t >0$.

Next, if $g\in \mathcal{B}_4$ and $b[g] \neq 0$,  then
by \eqref{RowA01+},
 \begin{align*}
|g^n(s/n)-e^{-s}
-a[g] n^{-1}s^2 e^{-s}|
\ge& |b[g_n]|s^4 e^{-s}-\frac{s^4}{2}d_0[g_n]\\
\ge&  n^{-2}s^3 \left(|b[g]|e^{-s}
-s g''''(0)\right), \qquad s >0.
\end{align*}
If
$$ s_0 = \min\left(1,\frac{|b[g]|}{2e g''''(0)}\right),$$
and $s \in (0,s_0],$ then
 \[
|g^n(s/n)-e^{-s}-a[g]n^{-1}s^2 e^{-s}|
\ge
 \frac{|b[g]|}{2e}\frac{s^3}{n^2}.
\]
So, by Theorem \ref{spmapping},
 we infer that
\begin{align*}
\|A^{-\alpha}(g^n(tA/n)-e^{-tA})\|\ge&
\sup_{\lambda\in \R_{+}}\,|\lambda^{-\alpha}(g^n(\lambda t/n)-
e^{-\lambda t}
-a[g]n^{-1} s^2 e^{-s})|\\
\ge& |t^\alpha s^{-\alpha}(g^n(s/n)-e^{-s}
-a[g]n^{-1} s^2 e^{-s})|\\
\ge& \frac{|b[g]|}{2e}\frac{s^{3-\alpha}t^\alpha }{n^2},
\end{align*}
where in the second line we have chosen  $\lambda=s/t$
 for  $s\in (0,s_0].$ This finishes the proof.
 \end{proof}
Similarly to the situation of Proposition \ref{Opt1},  it is straightforward to provide concrete examples of $A$'s satisfying
the assumptions of Proposition \ref{Opt2}, e.g. using  multiplication operators on $L^2(\mathbb R_+).$

\begin{remark}\label{BfM}
If $g\in \mathcal{B}_3$ is of the form
$g(z)=e^{-\varphi(z)}$, $\varphi\in \mathcal{BF}$,
then
$
b[g]=-\varphi'''(0)/6,
$
and $b[g] \neq 0$ if $g(z)\not=e^{-z}.$
On the other hand, if
\[
g_\alpha(z):=\frac{e^{-\alpha z}+e^{-(2-\alpha)z}}{2}, \qquad \alpha\in [0,1],
\]
then  $g_\alpha \in \mathcal{B}_3,$  $b[g_\alpha]=0,$
and $g_\alpha\in L^1(\R_{+})$ for  $\alpha\in (0,1]$.
\end{remark}

We finish the paper by illustrating the sharpness of our results in a particular situation of Euler's approximation formula.
Let  $g(z)=(1+z)^{-1}$. Note that $g\in \mathcal B_2,$  $g''(0)=2$ and $g \in L^2(\R_+).$
Now  Theorem \ref{ThA}, Corollary \ref{ThAxC} and
\eqref{calpha}
imply the following  first and second order Euler's approximation
formulas with rates  for sectorially bounded holomorphic $C_0$-semigroups.

\begin{thm}\label{Euler}
Let $-A$ be the generator of a sectorially bounded holomorphic
$C_0$-semigroup $(e^{-tA})_{t\ge 0}$ on $X.$
 Then for all $\alpha\in [0,1]$, $t>0$, and $n\in\N$,
\begin{equation}\label{b10E}
\left\|\left(1+\frac{t}{n}A\right)^{-n}x - e^{-tA}x \right\|\le
M_{2-\alpha} r_{\alpha,n} t^\alpha\|A^\alpha x\|,\qquad x\in \dom(A^\alpha),
\end{equation}
where
\[
r_{\alpha,n}=\frac{1}{2n}+\frac{1-2\alpha}{12n^2},\quad \alpha\in [0,1/2),\quad
r_{\alpha,n}=\frac{1}{2n},\quad \alpha\in [1/2,1],\quad n\in \N.
\]
Moreover, there exists an absolute constant $C>0$ such that for all
$\alpha\in [0,1]$, $t>0$, and $n\in\N,$
\begin{align}
\label{cggA0}
&\left\|\left(1+\frac{t}{n}A\right)^{-n}x - e^{-tA}x
-\frac{t^2A^2x}{2n}e^{-tA}x
\right\|\\
\le&
C[M_{3-\alpha}
+M_{4-\alpha}]\frac{t^\alpha}{n^2} \|A^\alpha x\|, \quad x\in \dom(A^\alpha).\notag
\end{align}
\end{thm}

Let us illustrate the sharpness of Theorem \ref{Euler}.
Observe that by Taylor series expansion, for fixed $t>0,$
\begin{align}
\frac{(1+t/n)^{-n}-e^{-t}}{1/n}=e^{-t}\left[ \frac{t^2}{2n} - \frac{t^3}{3n^2} + {\rm O}(1/n^3) \right], \qquad n \to \infty. \label{e1}
\end{align}
Apply Theorem  \ref{Euler} to the scalar function $e^{-t}$ when $\alpha$ is either $0$ or $1.$
If $\alpha=0$ then uniformly in $t\ge 0:$
\begin{align}\label{N0}
|(1+t/n)^{-n}- e^{-t}|
\le& \left(\sup_{s>0}\,s^2e^{-s}\right)\,
\left(\frac{1}{2n}+\mbox{O}(n^{-2})\right)\\
=&4e^{-2}\left(\frac{1}{2n}+\mbox{O}(n^{-2})\right), \quad n \to \infty,\notag
\end{align}
and if $\alpha=1$ then
\begin{align}\label{N1}
|(1+t/n)^{-n}- e^{-t}|
\le& t\left(\sup_{s>0}\,se^{-s}\right)\,
\left(\frac{1}{2n}+\mbox{O}(n^{-2})\right)\\
=&e^{-1}t\left(\frac{1}{2n}+\mbox{O}(n^{-2})\right), \quad n \to \infty.\notag
\end{align}

Thus, the abstract estimate (\ref{N0})  agrees with its numerical counterpart \eqref{e1} (where \eqref{N0} is a uniform version of \eqref{N1}).

\section{Appendix 1: Derivatives}

\begin{prop}\label{simpleN}
If $g\in \mathcal{B}_2$ and $g_n(t)=g^n(t/n), t\ge 0,$ then
\[
g_n(0)=1,\quad
g_n'(0)=-1,\quad
g_n''(0)=1+\frac{g''(0)-1}{n};
\]
if $g\in \mathcal{B}_3,$ then
\[
g_n'''(0)=\frac{1}{n^2}(-(n-1)(n-2)
-3(n-1)g''(0)+
g'''(0));
\]
and if $g\in \mathcal{B}_4,$ then
\begin{align*}
g_n''''(0)=&\frac{1}{n^3}((n-1)(n-2)(n-3)
+6(n-1)(n-2)g''(0)\\
+&3(n-1)(g''(0))^2
-4(n-1)g'''(0)
+g''''(0)).
\end{align*}
\end{prop}

\begin{proof}
The statement follows from the next explicit calculations:
\begin{align*}
g_n'(t)=&g^{n-1}(t/n)g'(t/n),\\
g_n''(t)=&((n-1)(g'(t/n))^2+
g(t/n)g''(t/n))\frac{g^{n-2}(t/n)}{n},\\
g_n'''(t)=&((n-1)(n-2)(g'(t/n))^3
+3(n-1)g(t/n)g'(t/n)g''(t/n)\\
+&
g^2(t/n)g'''(t/n))\frac{g^{n-3}(t/n)}{n^2},\\
g_n''''(t)=&((n-1)(n-2)(n-3)(g'(t/n))^4\\
+&6(n-1)(n-2)g(t/n)(g'(t/n))^2g''(t/n)\\
+3&(n-1)g^2(t/n)(g''(t/n))^2
+4(n-1)g^2(t/n)g'(t/n)g'''(t/n)\\
+&g^3(t/n)g''''(t/n))\frac{g^{n-4}(t/n)}{n^3}.
\end{align*}
\end{proof}

\begin{cor}\label{Corgn}
Let $g\in \mathcal{B}_3$ and the functionals $b$ and $d_0$ are given by \eqref{defab2} and \eqref{defd01}. respectively.
Then
\begin{equation}\label{CalA1}
b[g_n]=\frac{b[g]}{n^2},\qquad
b[g]=\frac{-2+3g''(0)+g'''(0)}{6},
\end{equation}
and, if $g\in \mathcal{B}_4$,
\begin{equation}\label{CalA3}
d_0[g_n]
=\frac{(g''(0)-1)^2}{4n^2}+\frac{-6+12g''(0)-3(g''(0))^2+4g'''(0)+g''''(0)}{12 n^3}.
\end{equation}
Moreover,
\begin{equation}\label{CalA2}
|b[g_n]|
 \le \frac{g''''(0)-1}{n^2},\qquad n\in \N,
\end{equation}
and
\begin{equation}\label{CalA4}
d_0[g_n]
\le \frac{g''''(0)-1}{n^2},\qquad n\in \N.
\end{equation}

If $g\in \mathcal{B}_4\cap L^k$ for some $k\in\N$ and $d_1$ is defined by \eqref{d1F},
then
there exists $C(g)>0$ such that
\begin{equation}\label{CalA5}
d_1[g_n]\le C(g) n^{-2},\qquad n\in \N.
\end{equation}
\end{cor}

\begin{proof}
First note that the formulas
(\ref{CalA1}) are (\ref{CalA3}) are direct consequences of
Proposition \ref{simpleN}.
To prove \eqref{CalA4} and \eqref{CalA5}, first recall that if  $g\in \mathcal{CM},$ then
\[
[g^{(k+1)}(z)]^2\le g^{(k)}(z)g^{(k+2)}(z),\qquad z>0, \quad k \in \N\cup \{0\},
\]
see \cite[Ch. XIII]{Mitr}. Thus, if in addition $g\in \mathcal{B}_3,$ then
\[
1\le g''(0)\le (g''(0))^2\le |g'''(0)|.
\]
Moreover,
for $g\in \mathcal{B}_4$, one has
\[
 1\le g''(0)\le 
(g''''(0))^{1/2}\le g''''(0),\qquad
|g'''(0)|\le (g''(0)g''''(0))^{1/2}\le g''''(0),
\]
hence for every $n \in \N,$
\[
|b[g_n]|
\le \frac{3(g''(0)-1)+(g'''(0)+1)}{6n^2}
\le \frac{4(g''''(0)-1)}{6n^2}\le \frac{(g''''(0)-1)}{n^2},
\]
and
\begin{align*}
 d_0[g_n]\le&
\frac{(g''(0)-1)^2}{4n^2}+\frac{-3+6g''(0)+4g'''(0)+g''''(0)}{12n^3}\\
\le& \frac{(g''''(0)-1)}{4n^2}+\frac{6(g''(0)-1)+(g''''(0)-1)}{12n^3}\\
\le& \frac{5(g''''(0)-1)}{6n^2}\\
\le& \frac{(g''(0)-1)}{n^2}, \qquad n \in \N.
\end{align*}

Finally, (\ref{CalA5}) follows from
$d_0[g]=-b[g]c_1[g]+c_2[g],$ (\ref{CalA2}) and
Theorem \ref{Asymp1}.
\end{proof}

\section{Appendix 2: Proof of Theorem \ref{Asymp1}}

Since for each $n \in \N,$ the mapping $[0,1] \ni \alpha \to c_\alpha[g_n]$ is continuous,
it suffices to prove \eqref{ContAs} only for $\alpha\in (0,1)$.

Fix
$\alpha\in (0,1)$. Then for each $n \in \N,$
\begin{align*}
c_\alpha[g_n]
=&-\frac{1}{\alpha\Gamma(2-\alpha)}\int_0^\infty (g_n(z)-e^{-z})
d z^{-\alpha}\\
=&\frac{1}{\alpha\Gamma(2-\alpha)}
\int_0^\infty \frac{g_n'(z)}{z^\alpha}\,dz+
\frac{1}{\alpha\Gamma(2-\alpha)}
\int_0^\infty e^{-z}z^{-\alpha}\,dz.
\end{align*}
Hence
\begin{equation}\label{cgnM}
c_\alpha[g_n]
=n^{1-\alpha}\frac{R_\alpha(n)}{\alpha\Gamma(2-\alpha)}
+\frac{1}{\alpha(1-\alpha)},
\end{equation}
where
\[
R_{\alpha}(n):=\int_0^\infty g^{n-1}(z)g'(z)z^{-\alpha}\,dz, \qquad n \in \N.
\]
Define
\[
f_\alpha(z):=\frac{z^\alpha(1-\alpha a(1-g(z))}{(1-g(z))^\alpha},\qquad a=\frac{g''(0)}{2},
\]
and write
\[
R_{\alpha}(n)=R_{0,\alpha}(n)+R_{1,\alpha}(n), \qquad n \in \N,
\]
where
\begin{align*}
R_{0,\alpha}(n):=\int_0^\infty g^{n-1}(z)g'(z)
z^{-\alpha}
f_\alpha(z)\,dz,
\end{align*}
and
\begin{align*}
R_{1,\alpha}(n):=\int_0^\infty g^{n-1}(z)g'(z)
z^{-\alpha}
[1-f_\alpha(z)]\,dz.
\end{align*}

We will estimate $R_{0,\alpha}$ and $R_{1,\alpha}$ separately.
To bound $R_{0,\alpha}(n)$, letting
 $s=g(z)\in (0,1)$, $z>0,$
and using \eqref{Class}, we obtain
\begin{align*}
R_{0,\alpha}(n)=&-\int_0^1 s^{n-1}
\frac{(1-\alpha a(1-s))}{(1-s)^\alpha}
\,ds
\\
=&-\frac{\Gamma(n)\Gamma(1-\alpha)}{\Gamma(n+1-\alpha)}
+\alpha a
\frac{\Gamma(n)\Gamma(2-\alpha)}{\Gamma(n+2-\alpha)}.
\end{align*}
Thus, employing \eqref{gamma1}, we have
\begin{align*}
Q_\alpha(n):=&\frac{n^{1-\alpha}}{\alpha\Gamma(2-\alpha)}R_{0,\alpha}(n)
+\frac{1}{\alpha(1-\alpha)}\\
=&\frac{n^{1-\alpha}}{\alpha\Gamma(2-\alpha)}
\left[-\frac{\Gamma(n)\Gamma(1-\alpha)}{\Gamma(n+1-\alpha)}
+\alpha a
\frac{\Gamma(n)\Gamma(2-\alpha)}{\Gamma(n+2-\alpha)}\right]
+\frac{1}{\alpha(1-\alpha)}\\
=&-
\frac{n^{1-\alpha}\Gamma(n)}{\alpha(1-\alpha)\Gamma(n+1-\alpha)}
+a\frac{n^{1-\alpha}\Gamma(n)}{\Gamma(n+2-\alpha)}
+\frac{1}{\alpha(1-\alpha)},
\end{align*}
so that
\begin{equation}\label{unifGT}
Q_\alpha(n)=-\frac{1}{2n}
+\frac{a}{n}
+\mbox{O}(n^{-2})=
\frac{g''(0)-1}{2n}+\mbox{O}(n^{-2}),\qquad n\to\infty,
\end{equation}
uniformly in  $\alpha\in (0,1)$

Let us now consider $R_{1,\alpha}(n).$
Note that in view of  \eqref{cgnM} and \eqref{unifGT},
if there exists $C(g)>0$ such that
\begin{equation}\label{forproof}
|R_{1,\alpha}(n)|\le
\frac{C(g)\alpha }{n^{3-\alpha}}, \qquad n \in \N,
\end{equation}
then \eqref{ContAs} follows. So let us proceed with the proof of \eqref{forproof}.
Let
\[
m(z):=\frac{1-g(z)}{z},\qquad z>0.
\]
Then
\begin{equation}\label{d00}
f_\alpha(z)=
m^{-\alpha}(z)\left[1-\alpha a zm(z)\right],
\end{equation}
and
\begin{equation}\label{d01}
f'_\alpha(z)
=-\alpha m^{-1-\alpha}(z)\left[
(1+(1-\alpha) a zm(z))m'(z)
+ a m^2(z)
\right].
\end{equation}
Since $g\in \mathcal{B}_2,$ we have  $m \in \mathcal{BM},$
\begin{align*}
m(0)=1,\quad  m(1)\le m(z)\le 1,&\qquad z\in [0,1],\\
(1-g(1))/z\le m(z)\le 1/z,& \qquad z\ge 1,
\end{align*}
and
\[
m'(0)=-a, \qquad |m'(z)|\le a,\qquad  z\ge 0.
\]
Therefore,  $f_\alpha \in C^2[0,\infty)$,
$
f_\alpha(0)=1$, and $f'_\alpha(0)=0.$
Moreover, by \eqref{d00} and \eqref{d01}, there exists $C_1(g)>0$ such that
\begin{equation}\label{SimEs}
|1-f_\alpha(z)|\le C_1(g)z^2,\qquad
|f'_\alpha(z)|\le C_1(g)\alpha z,\qquad z>0.
\end{equation}
Now, integrating by parts,
\begin{align*}
R_{1,\alpha}(n)=&\frac{1}{n}\int_0^\infty
z^{-\alpha}
(1-f_\alpha(z))\,d g^n(z)\\
=&\frac{1}{n}\int_0^\infty
g^n(z)\left[
\alpha z^{-1-\alpha}
(1-f_\alpha(z))
+z^{-\alpha}f'_\alpha(z)\right]\,dz.
\end{align*}
Hence, using \eqref{SimEs}, we  conclude that
\[
|R_{1,\alpha}(n)|
\le C_1(g)\frac{\alpha}{n}\int_0^\infty
g^n(z)z^{1-\alpha}\,dz,
\]
and Lemma \ref{nnk} implies \eqref{forproof}.
\section{Acknowledgments}
The authors would like to thank Satbir Singh Malhi, Reinhard Stahn,  Lasse Vuursteen
and two anonymous referees for useful remarks and suggestions
that significantly improved our presentation.

\end{document}

%% file: hatomakro.tex
\newcommand{\Wip}{\mathrm{A}_+^1}

\newcommand{\vanish}[1]{\relax}

\newcommand{\N}{\mathbb{N}}

\newcommand{\R}{\mathbb{R}}
\newcommand{\C}{\mathbb{C}}

\newcommand{\eM}{\mathrm{M}}

\newcommand{\Lap}{\mathcal{L}}

\newcommand{\Sum}[2][\relax]{%
 \ifx#1\relax \sideset{}{_{#2}}\sum
 \else \sideset{}{^{#1}_{#2}}\sum
 \fi}

\newcommand{\abs}[1]{\left| #1 \right|}
\renewcommand{\abs}[1]{\left\vert#1\right\vert}

\newcommand{\Ce}{\mathrm{C}}

\newcommand{\tpfeil}{\longmapsto}
\DeclareMathOperator{\dom}{dom}
\DeclareMathOperator{\ran}{ran}

\newcommand{\cls}[1]{\overline{#1}}

\DeclareMathOperator{\Lin}{\mathcal{L}}

\newcommand{\norm}[2][\relax]{%
   \ifx#1\relax \ensuremath{\left\Vert#2\right\Vert}
   \else \ensuremath{\left\Vert#2\right\Vert_{#1}}
   \fi}

\makeatletter
\newcommand{\sprod}[2]{\ensuremath{%
  \setbox0=\hbox{\ensuremath{#2}}
  \dimen@\ht0
  \advance\dimen@ by \dp0
  \left(\left.#1\rule[-\dp0]{0pt}{\dimen@}\,\right|#2\hspace{1pt}\right)}}
 \makeatother

\newcounter{aufzi}

\newcounter{aufzii}

\newcounter{aufziii}